\documentclass{amsart}
\usepackage{amsmath}
\usepackage{amssymb}
\usepackage{bbm}
\usepackage{mathdots}
\usepackage{dsfont}
\usepackage{graphicx}
\usepackage{tikz}
\usetikzlibrary{calc}
\usepackage{subcaption}
\usepackage{comment}

\def\RR{{\mathbb R}}

\def\ZZ{{\mathbb Z}}

\let\temp\phi
\let\phi\varphi
\let\varphi\temp

\def\ones{\mathbbm{1}}

\def\rank{\operatorname {rank}}

\theoremstyle{definition}
\newtheorem{definition}{Definition}[section]
\newtheorem{example}[definition]{Example}

\theoremstyle{plain}
\newtheorem{theorem}[definition]{Theorem}
\newtheorem{lemma}[definition]{Lemma}
\newtheorem{proposition}[definition]{Proposition}
\newtheorem{corollary}[definition]{Corollary}
\newtheorem*{corollary*}{Corollary}

\newtheorem{remark}[definition]{Remark}
\newtheorem{observation}[definition]{Observation}

\usepackage[draft]{fixme}
\fxsetup{theme=color, mode=multiuser, noinline}
\FXRegisterAuthor{rt}{RT}{\color{blue}Rekha}
\FXRegisterAuthor{ss}{TD}{\color{teal}Stefan}

\begin{document}

\title{Conformally rigid graphs}

\author[]{Stefan Steinerberger}
\address{Department of Mathematics, University of Washington, Seattle, WA 98195, USA} 
\email{steinerb@uw.edu}

\author[]{Rekha R. Thomas}
\address{Department of Mathematics, University of Washington, Seattle, WA 98195, USA} 
\email{rrthomas@uw.edu}

\begin{abstract}
 Given a finite, simple, connected graph $G=(V,E)$ with $|V|=n$,  we consider the associated graph Laplacian matrix $L = D - A$ with eigenvalues $0 = \lambda_1 < \lambda_2 \leq \dots \leq \lambda_n$. One can also consider the same graph equipped with positive edge weights $w:E \rightarrow \mathbb{R}_{> 0}$ normalized to $\sum_{e \in E} w_e = |E|$ and the associated weighted Laplacian matrix $L_w$. 
We say that $G$ is \textit{conformally rigid} if constant edge-weights
 maximize the second eigenvalue $\lambda_2(w)$ of $L_w$ over all $w$, and minimize $\lambda_n(w')$ of $L_{w'}$ over all $w'$, i.e., for all $w,w'$, 
 $$ \lambda_2(w) \leq \lambda_2(\ones) \leq \lambda_n(\ones) \leq \lambda_n(w').$$
Conformal rigidity requires an extraordinary amount of structure in $G$. Every edge-transitive graph is conformally rigid. We prove that every distance-regular graph, and hence every strongly-regular graph, is conformally rigid. Certain special graph embeddings can be used to characterize conformal rigidity. Cayley graphs can be conformally rigid but need not be, we prove a sufficient criterion. We also find a small set of conformally rigid graphs that do not belong into any of the above categories; these include the Hoffman graph, the crossing number graph 6B and others. Conformal rigidity can be certified via semidefinite programming, we provide explicit examples. 
\end{abstract}

\maketitle

\section{Introduction}
\subsection{Introduction}
The purpose of this paper is to introduce a property, which we call {\em conformal rigidity}, that a finite, simple graph $G=(V,E)$ may or may not have. It is the discrete analogue of a property of interest in spectral geometry and can only be expected in the presence of structure. Roughly put, it says that among all possible ways one could weigh the edges, constant weight edges are already, in a suitable sense, the canonical choice. We start by defining things more formally. Let $G=(V,E)$ be a finite, simple, connected graph on $n$ vertices. A classical object associated to $G$ is its graph Laplacian $L = D-A$, where $D \in \mathbb{R}^{n \times n}$ is the diagonal matrix collecting the degrees of the vertices and $A \in \left\{0,1\right\}^{n \times n}$ is the adjacency matrix of $G$. The matrix $L$ is positive semidefinite (psd) and has eigenvalues
$$ 0 = \lambda_1 < \lambda_2 \leq \dots \leq \lambda_n$$
which are fundamental objects in spectral graph theory. The theory generalizes to the case where edges are weighted $w_{ij} > 0$; the weighted Laplacian $L_w \in \RR^{n \times n}$ is the psd matrix defined as:
\begin{align}
    (L_w)_{ij} := \left\{ \begin{array}{rl} -w_{ij} & \textup{ if } 
    (i,j) \in E \\ 
    \sum_{j \sim i} w_{ij} & \textup { in the } (i,i) \textup{ entry}\\
    0 & \textup{ otherwise } \end{array} \right.
\end{align}
with eigenvalues $0 = \lambda_1(w) < \lambda_2(w) \leq \dots \leq \lambda_n(w)$. We use $j \sim i$ to denote that $(i,j) \in E$. We only work with undirected graphs, so there is no difference between $j \sim i$ and $i \sim j$. The constant vector $\ones := (1,1,\ldots,1)$ spans the eigenspace of $\lambda_1(w) = 0$. 
By the Rayleigh-Ritz Theorem \cite[\S 4.2]{horn-johnson}, we can write $\lambda_2(w)$ as
$$ \lambda_2(w) = \min_{\sum_{v \in V} f(v) = 0} \frac{\sum_{(i,j) \in E} w_{ij} (f(i) - f(j))^2}{ \sum_{v \in V} f(v)^2}$$
where the minimum ranges over all $f:V \rightarrow \mathbb{R}$ that are not identically 0. 
The condition $\sum_{v \in V} f(v) = 0$ forces orthogonality to the constant vector $\ones$.  Likewise, and without orthogonality conditions, we have
$$ \lambda_n(w) = \max_{f:V \rightarrow \mathbb{R}} \frac{\sum_{(i,j) \in E} w_{ij} (f(i) - f(j))^2}{ \sum_{v \in V} f(v)^2}.$$
It is well known that $G$ is connected if and only if $\lambda_2(w) > 0$. In fact, $\lambda_2(w)$ measures the connectivity of a graph: larger values correspond to a more connected graph. The largest eigenvalue $\lambda_n(w)$ is more intricate but corresponds roughly to the largest possible oscillation that can occur among the functions of 
the graph. 

\begin{center}
    \begin{figure}[h!]
    \begin{tikzpicture}
            \node at (-4,0) {\includegraphics[width=0.3\textwidth]{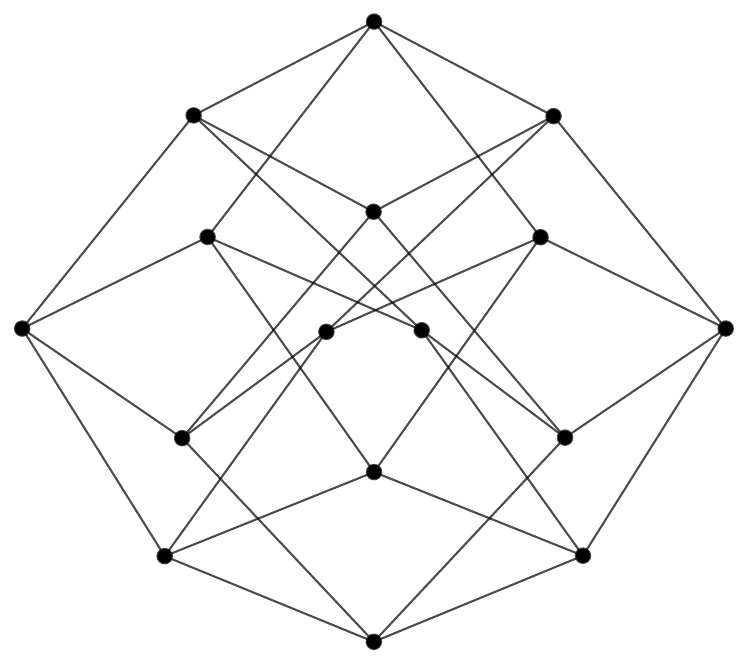}};
        \node at (0,0) {\includegraphics[width=0.25\textwidth]{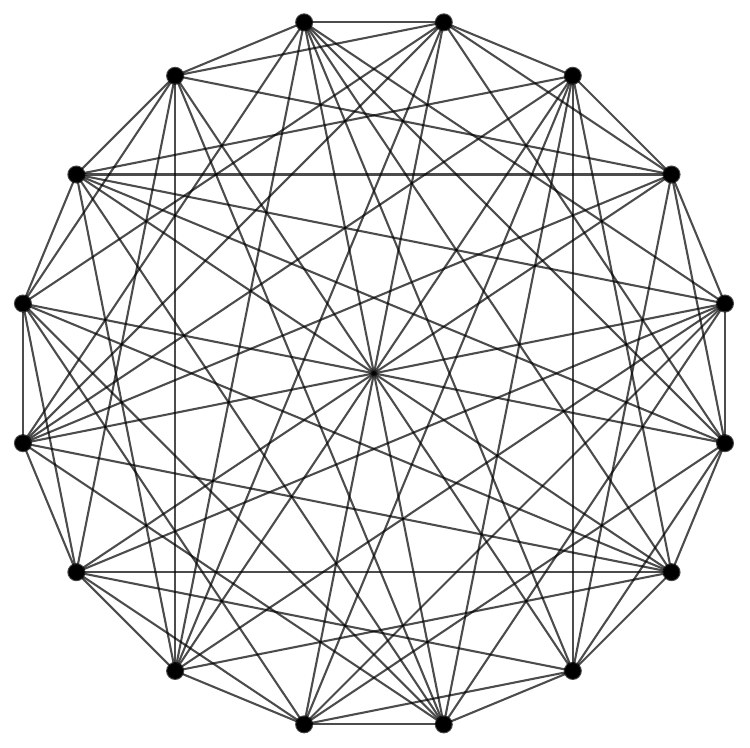}};
        \node at (4,0) {\includegraphics[width=0.25 \textwidth]{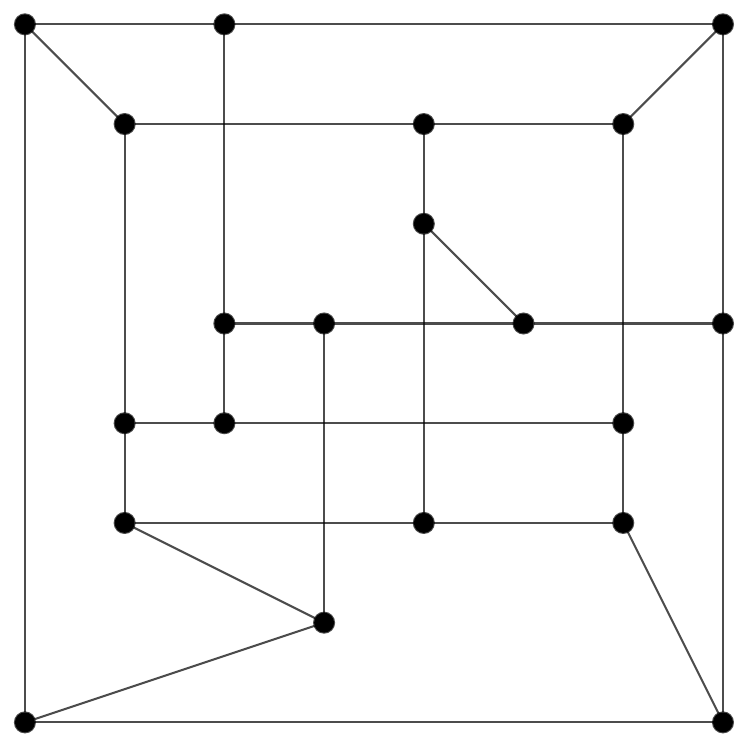}};
    \end{tikzpicture}
    \caption{Three conformally rigid graphs. Left: Hoffman graph. Middle: complement of the Shrikhande graph. Right: CNG 6B.}
    \label{fig:conformally rigid}
    \end{figure}
\end{center}
\vspace{-20pt}

It is a natural question how, for any given graph $G=(V,E)$, one should choose weights $w$ to maximize $\lambda_2(w)$ or to minimize $\lambda_n(w)$, respectively.  Since the eigenvalues scale linearly with $w$, we require both $w_{ij} \geq 0$ and the normalization $$\sum_{(i,j) \in E} w_{ij} = |E|$$
which we do throughout the  paper.  
Maximizing $\lambda_2(w)$ can be understood in several ways, including creating the `densest' network, while minimizing $\lambda_n(w)$ is akin to a smoothness condition on the weights. 
It is a way of `regularizing' the graph, finding its most symmetric representation. We are interested in graphs for which these optimization procedures cannot improve on the unweighted case.
    \begin{definition} \label{def:conformally rigid}  We say a graph $G=(V,E)$ is \textit{conformally rigid} if, for all edge weights $w,w':E \rightarrow \mathbb{R}_{\geq 0}$ normalized to $\sum_{(i,j) \in E} w_{ij} = |E|$ and $\sum_{(i,j) \in E} w_{ij}' = |E|$, 
    $$ \lambda_2(w) \leq \lambda_2(\ones) \leq \lambda_n(\ones) \leq \lambda_n(w').$$
    \end{definition}

    \begin{figure}[h!]
        \centering
\begin{tikzpicture}
    \draw [thick] (0,0) -- (6,0);
  \draw [ultra thick] (2,0) -- (4,0);
    \draw [ultra thick] (1,-0.1) -- (1,0.1);
    \draw [ultra thick] (2,-0.1) -- (2,0.1);
    \draw [ultra thick] (4,-0.1) -- (4,0.1);
 \draw [ultra thick] (5,-0.1) -- (5,0.1);
 \node at (1, -0.4) {$\lambda_2(w)$};
  \node at (2, -0.4) {$\lambda_2(\ones)$};
   \node at (4, -0.4) {$\lambda_n(\ones)$};
    \node at (5, -0.4) {$\lambda_n(w')$};
\end{tikzpicture}        
\caption{If a graph is conformally rigid, then the the interval spanned by the spectrum of the Laplacian, for any assignment of weights (fixing the total sum), always contains $[\lambda_2(\ones), \lambda_n(\ones)]$. }
        \label{fig:spectrum}
    \end{figure}
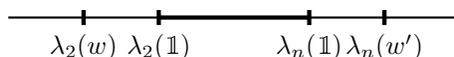
    This corresponds, in a suitable sense, to graphs which are already so symmetric that varying edge weights cannot further increase $\lambda_2$ or decrease $\lambda_n$. The definition can be phrased in terms of the spectrum of the Laplacian  matrix 
as in Fig.~\ref{fig:spectrum}. An alternative viewpoint is as follows: by changing the weights, we change the eigenvalues. As mentioned already, there is a linear scaling, doubling all the weights doubles all the eigenvalues, which makes it natural to normalize the sum of the weights. Conformal rigidity is concerned with the setting where the largest attainable second eigenvalue over all possible weights is simply the eigenvalue of the unweighted graph itself, meaning
$$\sup_{w >0} \frac{\lambda_2(w)}{ \sum_{(i,j) \in E} w_{ij}}  = \sup_{w >0} \inf_{f:V \rightarrow \mathbb{R} \atop \sum_{v \in V} f(v) = 0} \frac{\sum_{(i,j) \in E} w_{ij} (f(i) - f(j))^2}{\left( \sum_{(i,j) \in E} w_{ij}\right) \left(\sum_{v \in V} f(v)^2\right)} = \lambda_2(\ones)$$
and where, simultaneously, the same is true for the largest eigenvalue $\lambda_n(w')$
$$\inf_{w' > 0} \frac{\lambda_n(w')}{ \sum_{(i,j) \in E} w'_{ij}}  = 
\inf_{w' >0} \sup_{f:V \rightarrow \mathbb{R}} \frac{\sum_{(i,j) \in E} w'_{ij} (f(i) - f(j))^2}{\left( \sum_{(i,j) \in E} w'_{ij}\right) \left(\sum_{v \in V} f(v)^2\right)} = \lambda_n(\ones).$$
These equations show the implicit mathematical structure of the question: there is an entire additional optimization problem sitting on top of the Rayleigh-Ritz formulations of $\lambda_2$ and $\lambda_n$ which are optimization problems in themselves.\\

Let us quickly consider some examples. It may not come as a surprise that the complete graph $K_n$ is conformally rigid. An elementary proof can be found in \S~\ref{proofs:Kn}.

\begin{proposition} \label{prop:Kn}
 The complete graph $K_n$ is conformally rigid. Moreover, the only $w$ for which $ \lambda_2(w) = \lambda_2(\ones)$ and 
$\lambda_n(\ones) = \lambda_n(w'),$ is $w = w'=\ones$. 
\end{proposition}

Conformal rigidity does \textit{not} imply that 
$\ones$ is the unique weight vector that maximizes 
$\lambda_2(w)$ and minimizes $\lambda_n(w')$: a conformally rigid graph $G$ may have non-constant weights $w,w'$ with
$\lambda_2(w) = \lambda_2(\ones)$, an example being a circulant graph in Example \ref{ex:C18}, or $\lambda_n(w') = \lambda_n(\ones)$ with an example being the Haar Graph 565, \S\ref{subsec:symmetry}.

\subsection{Motivation} Our definition of conformal rigidity in graphs has a number of different motivations, we describe three below.

\subsubsection{Conformal Optimization.} The first motivation comes from 
(continuous) spectral geometry. 
Let $(M,g)$ be a smooth, compact Riemannian manifold of dimension $n \geq 2$. The Laplace-Beltrami operator $-\Delta$ has a discrete (infinite) sequence of eigenvalues $0 = \lambda_1 < \lambda_2 \leq \dots$
and it is a classical problem to understand which manifold minimizes or maximizes $\lambda_k$ under suitable conditions. Let us now consider changes of the underlying metric $g$ and introduce the conformal class of the metric  
$$[g] = \left\{h = e^{2u} g, u \in C^{\infty}(M) \right\}.$$
This corresponds to local rescalings of the metric that preserve angles: one can blow up certain parts of the manifold and one can shrink other parts but one cannot change the local conformal structure. We can now define
$$ \Lambda_k(M, [g]) = \sup_{ h \in [g]} \lambda_k(h) \cdot \mbox{vol}_{h}(M)^{\frac{2}{n}}.$$
The volume term counters the fact that this procedure may increase or decrease volume (and Laplacian eigenvalues are sensitive to that).
Alternatively, one could only consider metrics for which $ \mbox{vol}_{h}(M)^{\frac{2}{n}}$ is of a fixed size; this would correspond to our edge normalization $\sum_{(i,j) \in E} w_{ij} = |E|$.  Korevaar \cite{korevaar} proved that this supremum is finite for each value of $k$. Very much as in our setting, metrics that attain the supremum would correspond to the `most symmetric' realization of the manifold in its conformal class. As a motivating result, we recall a classic result of Hersch \cite{hersch} and El Soufi - Ilias \cite{el} showing that
$$ \Lambda_2(\mathbb{S}^n, [dx]) = n \cdot \omega_n^{\frac{2}{n}},$$
where $\omega_n$ is the volume of the standard sphere and equality is attained if and only if the sphere is round. Phrased differently, the round sphere is already the most symmetric realization.
Conformally rigid graphs are, in this sense, the discrete analogues of spheres in this context: their unweighted version is already spectrally extremal among all possible ways of distributing weights over the edges.

\subsubsection{Frame Bounds.} Another reason for why we are interested in conformally rigid graphs is the following basic fact corresponding to the spectrum of the complete graph $K_n$: given any list of $n$ numbers $x_1, \dots, x_n \in \mathbb{R}$, we have
$$ \sum_{i=1}^{n} x_i = 0 \implies \sum_{1\leq i<j\leq n} (x_i - x_j)^2 = n \sum_{i = 1}^{n} x_i^2.$$
One way of phrasing this result is as follows: it is possible deduce the $\ell^2-$norm of $(x_1, \dots, x_n)$ just from knowing all the pairwise differences $x_i - x_j$ (up to, of course, additive constants which cannot be recovered). There is a natural question: if we no longer have access to all pairwise differences but only some (encoded by  a graph), how accurately can we recover the norm? This is determined by the spectrum
$$ \sum_{i=1}^{n} x_i = 0 \implies \qquad \lambda_2\sum_{i = 1}^{n} x_i^2 \leq  \sum_{(i,j) \in E}^{} (x_i - x_j)^2 \leq \lambda_n \sum_{i = 1}^{n} x_i^2,$$
and can be interpreted as frame bounds, with $\lambda_n/\lambda_2$ being the condition number. If we are restricted to working with a fixed set of differences, a fixed $E$, then it is a natural question whether one can improve the accuracy of recovery by adding weights. Conformally rigid graphs have the property that allowing weights can neither increase $\lambda_2$ nor decrease $\lambda_n$, the graphs are already optimal in this regard. One way of stating this precisely is via the following Proposition which follows immediately from the previously mentioned sup/inf characterization.

\begin{proposition} \label{prop:weight}
    Let $G=(V,E)$ be conformally rigid. Then, for all possible (positive) edge weights $w_{ij}$ there always exists a function $f:V \rightarrow \mathbb{R}$ with mean value 0, meaning $\sum_{v \in V} f(v) = 0$, which oscillates little in the sense of
  $$\sum_{(i,j) \in E} w_{ij} (f(i) - f(j))^2 \leq  \frac{\lambda_2(\ones)}{|E|}\left( \sum_{(i,j) \in E} w_{ij}\right) \left(\sum_{v \in V} f(v)^2\right).$$ 
  Moreover, there exists a function $f:V \rightarrow \mathbb{R}$ which oscillates a lot insofar as
    $$\sum_{(i,j) \in E} w_{ij} (f(i) - f(j))^2 \geq \frac{\lambda_n(\ones)}{|E|}\left( \sum_{(i,j) \in E} w_{ij}\right) \left(\sum_{v \in V} f(v)^2\right).$$
\end{proposition}
Note that when all the weights are $w_{ij} =1$, then the Proposition is merely a form of the classical Minimax formulation for eigenvalues. The new ingredient, given by conformal rigidity, is that we allow for general weights.

\subsubsection{Sparsification.} A final motivation comes from a recently proposed framework for Graph Sparsification due to Babecki and the authors \cite{sparse}. In this framework, the hypercube graph cannot be further sparsified. We were wondering how general this result is which is what lead us to the notion of conformally rigid graphs. If (up to scaling) $\ones$ is the unique weight that maximizes $\lambda_2(w)$, then $G$ cannot be sparsified in any setting of \cite{sparse}. This led us to wonder if all Cayley graphs have this property. As we will see in this paper, not all Cayley graphs have this property. This line of investigation led us to the notion of conformal rigidity. It is an interesting question to understand which conformally rigid graphs have unique maximizers for one or both of $\lambda_2(w)$ and $\lambda_n(w)$. 

\subsection{Related results}

In this section we collect together various results in the literature that 
are close to the notion of conformal rigidity. 

\subsubsection{The optimization perspective} \label{subsec:opt}
The problem of maximizing $\lambda_2(w)$ subject to $\sum w_{ij} = |E|$ appears to have been considered first by Fiedler \cite{fiedler1990trees, fiedler1993minmax}
and references therein. Fiedler calls the maximum value of $\lambda_2(w)$ the {\em absolute connectivity} of the graph.  Finding weights that maximize $\lambda_2(w)$ subject to a general weighted  inequality of the form $\sum d_{ij}^2 w_{ij} \leq 1$ underlies the problem of finding the fastest mixing Markov process on $G$. This problem was studied by Sun-Boyd-Xiao-Diaconis 
\cite{sun-boyd-xiao-diaconis} where they modeled it as a {\em semidefinite program} which is a type of convex optimization problem. We will expand on this in \S~\ref{sec:SDP}.
Boyd-Diaconis-Parrilo-Xiao \cite{boyd-diaconis-parrilo-xiao} exploited the symmetries of the graph to speed up computations. 
We will adopt the methods in both \cite{boyd-diaconis-parrilo-xiao} and \cite{sun-boyd-xiao-diaconis} to certify the conformal rigidity of graphs.
The semidefinite program viewpoint of maximizing $\lambda_2(w)$ was also studied in \cite{goering-helmberg-wappler} where they connect it to the problem of separators in $G$. Both 
\cite{goering-helmberg-wappler} and \cite{sun-boyd-xiao-diaconis} interpret this problem as a graph embedding problem.
The problem of minimizing $\lambda_n(w)$ is also a semidefinite program with interpretations in terms of graph embeddings and partitions. 
The difference $\lambda_n(w)-\lambda_2(w)$ of the Laplacian $L_w$ is called the {\em spectral width} of $L_w$.  The problem of  minimizing the spectral width of $G$ is 
\begin{align} \label{eq:prob spectral width}
    \min \left\{ \lambda_n(w) - \lambda_2(w) \,:\, \sum w_{ij} = |E|, \,\,\,w \geq 0 \right\},  
\end{align}
and is also a semidefinite program   with connections to graph embeddings and separators \cite{goering-helmberg-reiss-spectralwidth}. The spectral width  bounds the {\em uniform sparsest cut} in $G$ since 
$$\frac{1}{n} \lambda_2(w) \leq 
\frac{w(S, V \setminus S)}{|S| |V \setminus S|}
\leq \frac{1}{n} \lambda_n(w)$$
where $w(S, V \setminus S)$ is the weight of the cut in $G$ given by the partition $V = S \cup (V \setminus S)$. \\

\textbf{Remark.} The spectral width minimization problem looks similar to 
conformal rigidity and so we pause to point out an important difference. 
In our setup we are solving
\begin{align} \label{eq:our prob}
    \min \left\{ \lambda_n(w) - \lambda_2(w') \,:\, \ones^\top w = \ones^\top w'=|E|, \,\,\,w,w'\geq 0 \right\}. 
\end{align}
where $w$ does not have to equal $w'$. If $G$ is conformally rigid then $\ones=w=w'$ is an optimal solution to \eqref{eq:our prob} and the optimal value of \eqref{eq:our prob} is the smallest spectral width of $G$. However, if 
$\ones$ is an optimal solution of 
\eqref{eq:prob spectral width}, $G$ may not be conformally rigid since \eqref{eq:our prob} allows for different weights in $\lambda_2(w)$ and $\lambda_n(w')$ potentially making their difference smaller than 
$\lambda_n(\ones) - \lambda_2(\ones)$.

\subsubsection{The discrete Wirtinger inequality}
Another topic that is related to, and partially inspired our results, is the discrete Wirtinger inequality proved by Fan, Taussky and Todd \cite{fan} in 1955 (see also an earlier paper of Schoenberg \cite{schoen}).
\begin{theorem}[Fan-Taussky-Todd \cite{fan}] If $x_1, \dots, x_{n+1}$ are real numbers so that $x_1 = x_{n+1}$ and $\sum_{i=1}^{n} x_i = 0$, then
$$ \sum_{i=1}^{n} (x_i - x_{i+1})^2 \geq 4 \sin^2\left(\frac{\pi}{n}\right) \sum_{i=1}^{n} x_i^2$$
with equality unless $x_i = A \cos(2\pi i/n) + B \sin(2 \pi i/n)$.
\end{theorem}
In modern language, this inequality is a way of describing the second eigenvalue $\lambda_2$ of the graph Laplacian on the cycle graph $C_n$ with $n$ vertices.
It is clear that this argument, interpreting the quadratic form as that of the Laplacian of $C_n$ and the constant as the spectral gap, will generalize to weights, but the arising matrix computations are not  trivial \cite{iz, mil, zhang}. We present a somewhat dual statement of the usual discrete Wirtinger inequality with weights.

\begin{corollary}
    The cycle graph $C_n$ is conformally rigid. In particular, if $w_1, \dots w_n$ are positive edge weights, then there always exist real numbers $x_1, \dots, x_{n+1}$ so that $x_1 = x_{n+1}$ as well as $\sum_{i=1}^{n} x_i = 0$ and
    $$ \sum_{i=1}^{n} w_i (x_i - x_{i+1})^2 \leq 4 \sin^2\left(\frac{\pi}{n}\right)\left(\frac{1}{n} \sum_{i=1}^{n} w_i\right) \sum_{i=1}^{n} x_i^2$$
\end{corollary}
One way of phrasing this Corollary is that the largest constant in the discrete Wirtinger inequality is attained if all the weights are the same: one cannot hope to get an improved discrete Wirtinger inequality by a clever choice of weights. Note that this is simply Proposition \ref{prop:weight} applied to the cycle graph. One also obtains a dual result (from the second part of Proposition \ref{prop:weight}): for any choice of weights there exists a function that oscillates at least as rapidly as the most oscillating function for equal weights, however, this case seems to be less relevant in applications.

\subsection{Organization of the paper}
In \S~\ref{sec:main results} we state our main results beginning with a summary. In \S~\ref{sec:SDP} we explain how semidefinite programming can be used to both test and certify conformal rigidity. We also explain how the automorphisms of the graph can sometimes drastically reduce the computational effort and also provide structural insight. \S~\ref{sec:embeddings} discusses the connection between conformal rigidity and graph embeddings into Euclidean space. Section~\ref{sec:proofs} contains the proofs of all our main results. Section~\ref{sec:sdp certificates} surveys various ways of proving that an explicitly given graph is conformally rigid via semidefinite programming certificates.

\section{Main Results} \label{sec:main results}
\subsection{Summary}
Our main result is that the class of conformally rigid graphs form an interesting and highly nontrivial collection of graphs. 
\begin{enumerate}
    \item All edge-transitive graphs are conformally rigid and this is true for a simple reason (see \S 2.1). The smallest example of a conformally rigid graph that is not edge-transitive (that we know of) is the Hoffman graph on 16 vertices.
   \item All distance-regular graphs and, as a special case, all strongly-regular graphs are conformally rigid.
    \item Some Cayley graphs are conformally rigid and others are not. We provide a sufficient condition for Cayley graphs to be conformally rigid (Theorem \ref{thm:Cayley}). The smallest Cayley graph that we found that is conformally rigid but \textit{not} edge-transitive is a circulant graph on 18 vertices (see \S 2.2). 
    \item As an application of the criterion for Cayley graphs, we found (empirically) that there seems to be a large number of circulant graphs that are conformally rigid (see \S 2.3). They seem to have a rich structure. There are also infinite families of circulants that are not conformally rigid. 
    \item Conformal rigidity can be characterized in terms of the existence of 
    special spectral embeddings of the graph on the eigenspaces of $\lambda_2(\ones)$ and $\lambda_n(\ones)$. These embeddings are {\em edge-isometric}, meaning that that 
    all edge lengths must be equal (see \S~\ref{sec:embeddings}).
    \item For any given graph $G=(V,E)$ it can be decided whether it is conformally rigid with a finite amount of computation by phrasing it as a semidefinite program. We explain how this is done and use it to verify conformal rigidity of a number of seemingly isolated examples that do not seem to fall into any of the above categories
    \begin{itemize}
        \item the Hoffman graph on 16 vertices.
         \item the crossing number graph 6B on 20 vertices \cite{crossing}.
        \item the distance-2 graph of the Klein graph on 24 vertices.
        \item the $(20,8)$ accordion graph \cite{accordion} on 40 vertices.
    \end{itemize}
\end{enumerate}

We summarize our results in the figure below. A solid arrow implies containment while a dashed arrow implies partial containment.

\begin{center}
    \begin{figure}[h!]
        \centering
    \begin{tikzpicture}
    \node at (0,0) {distance-regular};
        \node at (-2,1) {distance-transitive};
    \node at (2,1) {strongly regular};
    \draw[->] (2,0.8) -- (1.3, 0.2);
        \draw[->] (-2,0.8) -- (-1.3, 0.2);
    \draw [] (-1.75,-1.2) -- (1.8,-1.2) -- (1.8,-0.8) -- (-1.75, -0.8) -- (-1.75, -1.2);
    \node at (0,-1) {\textsc{ conformally rigid}};
      \draw[->] (0,-0.3) -- (0,-0.7);
      \node at (-3,0) {arc-transitive};
          \draw[->] (-2,0.8) -- (-2.5, 0.2);
            \node at (-3,-2) {vertex- and edge-transitive};
                \draw[->] (-3,-0.2) -- (-3, -1.7);
                    \draw[->] (-0.8,-2) -- (-0.2, -2);
                        \draw[->] (0.5, -1.7) -- (0.5, -1.3);
           \node at (1,-2) {edge-transitive};
        \node at (4, 0) {Cayley};
        \draw[->, dashed] (3.8, -0.2) -- (2,-0.8);
        \node at (3.8, -0.6) {\tiny Theorem \ref{thm:Cayley}};
        \node at (-0.5, -1.5) {\tiny Proposition \ref{prop:edge-transitive}};
      \node at (-0.9, -0.5) {\tiny Theorem \ref{thm:distanceregular}};
    \end{tikzpicture}
        \caption{Summary of our main results.}
        \label{fig:summary}
    \end{figure}
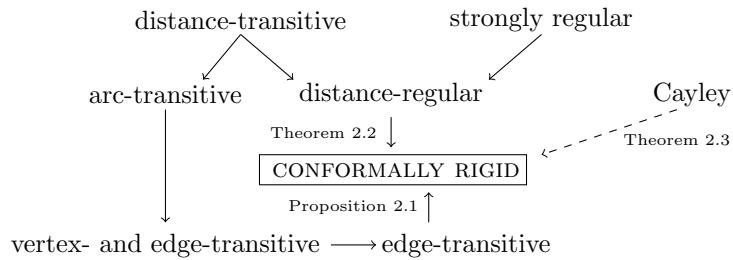
\end{center}

Throughout this paper we assume that the vertex set of $G=(V,E)$ is enumerated by the first $n$ integers, $V = [n]$, and that all graphs are connected, undirected and simple. Also, for simplicity, we sometimes write $\lambda_2$ for $\lambda_2(\ones)$ and $\lambda_n$ for $\lambda_n(\ones)$.

\subsection{Edge-transitive graphs are conformally rigid}

Recall that an {\em automorphism} of $G=(V,E)$ is a permutation $\pi$ of its vertices so that $(i,j) \in E$ if and only if $(\pi(i),\pi(j)) \in E$. The set of all automorphisms of $G$, denoted $\textup{Aut}(G)$ and called the 
{\em automorphism group} of $G$, is a subgroup of $S_n$. 
The graph $G$ is {\em vertex-transitive} if all vertices of $G$ lie in a single orbit of $\textup{Aut}(G)$. The action of $\textup{Aut}(G)$ on $V$ induces an action on $E$ via $(i,j) \mapsto (\pi(i),\pi(j))$. The graph $G$ is {\em edge-transitive} if all edges of $G$ lie in a single orbit 
of $\textup{Aut}(G)$. Complete graphs are edge-transitive and so are many other common families of graphs such as cycles, complete bipartite graphs and edge graphs of hypercubes. All distance-transitive graphs are edge-transitive (see Fig.~\ref{fig:summary}). The following is a first observation.

\begin{proposition} \label{prop:edge-transitive}
Every edge-transitive graph is conformally rigid.
\end{proposition}

The proof of this Proposition is a simple corollary of the technique of symmetry reduction (see for example~\cite{boyd-diaconis-parrilo-xiao}) which we recall in \S~\ref{sec:SDP}. All proofs of results in this section can be found in \S~\ref{sec:proofs}.
One way of interpreting Proposition~\ref{prop:edge-transitive} is as follows: edge-transitivity is a strong symmetry property that essentially says that all edges in $G$ are the ``same'' which points toward weighting them equally. Therefore, in some sense, the most interesting conformally rigid graphs are the ones that are \textit{not} edge-transitive since they are conformally rigid for a different reason. In the rest of the paper, we will focus on such graphs.

\subsection{Distance-regular graphs}
\label{subsec:distance regular}
Recall that a finite graph $G=(V,E)$ is distance-regular if for any two vertices $u,v$, the number of vertices that is distance $i$ from $u$ and distance $j$ from $v$ only depends on $i,j$ and the distance $d(u,v)$.

\begin{theorem} \label{thm:distanceregular}
Every distance-regular graph is conformally rigid.
\end{theorem}

  Note that, in particular, every strongly-regular graph is distance-regular. Strongly-regular graphs are distance-regular graphs of diameter two (where diameter is the largest distance between any pair of points).   For example, the Shrikhande graph and its complement (see Fig.~\ref{fig:conformally rigid}) are strongly-regular. 
  The proof of Theorem~\ref{thm:distanceregular} uses structural theory of distance-regular graphs (see \cite{godsil-alg-comb-book}) and the duality theory of semidefinite programs which are developed in \S~\ref{sec:SDP} and \S~\ref{sec:embeddings}.

\begin{center}
    \begin{figure}[h!]
    \begin{tikzpicture}
            \node at (-4,0) {\includegraphics[width=0.28\textwidth]{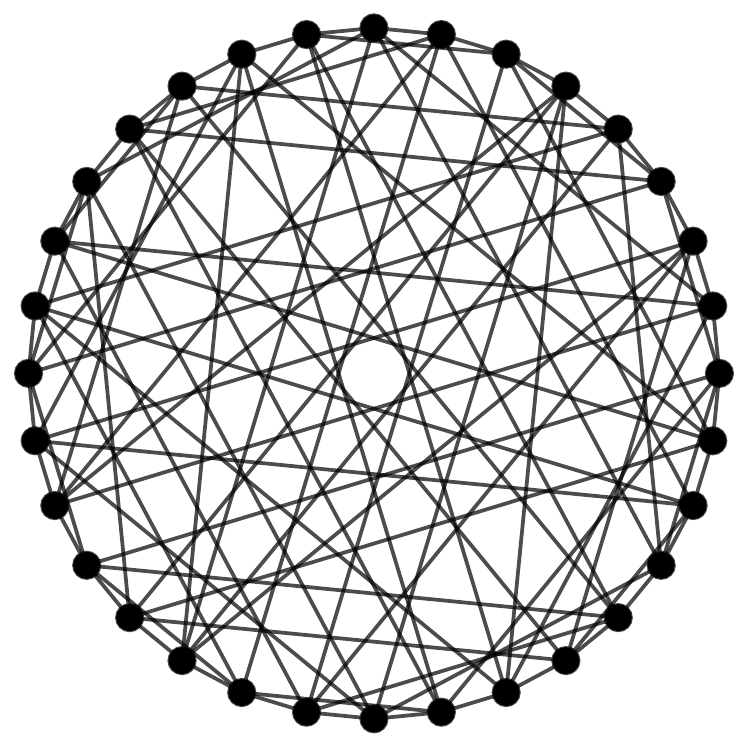}};
        \node at (0,0) {\includegraphics[width=0.28\textwidth]{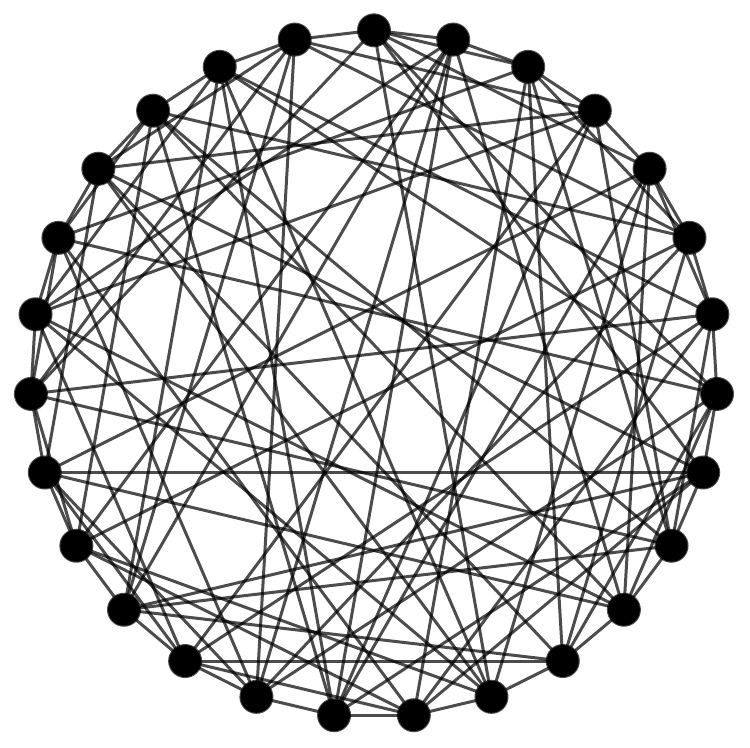}};
        \node at (4,0) {\includegraphics[width=0.28 \textwidth]{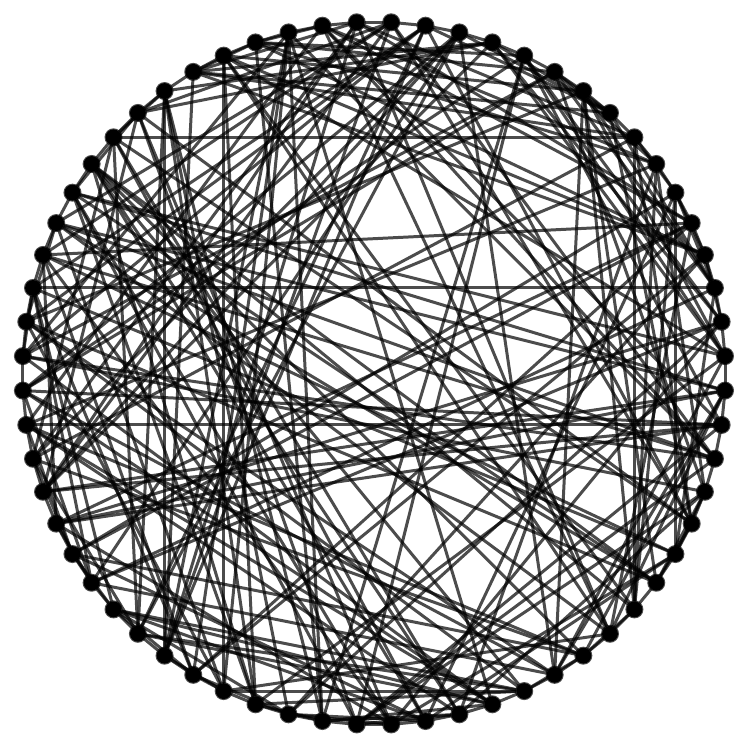}};
    \end{tikzpicture}
    \caption{Three conformally rigid graphs that are distance-regular but neither 
    strongly-regular nor edge-transitive. Left: a bipartite $(0,2)-$graph. Middle: a generalized quadrangle graph. Right: Doob graph $D(1,1)$.}
    \label{fig:disreg}
    \end{figure}
\end{center}
\vspace{-10pt}

\subsection{Cayley Graphs}
A natural class of graphs with a great deal of symmetry are Cayley graphs. The Cayley graph 
$\Gamma(G,S)$ has vertex set $V = G$ where $G$ a finite group and the edge set $ E = \left\{(g, s \circ g): g \in G, s \in S \right\}$ where $S \subset G$ is a {\em generating set}. 
Since we only deal with undirected graphs, we want the edge set to be symmetric which is achieved by requiring a symmetry in the generating set $S$, meaning $s \in S \implies s^{-1} \in S$. We always assume this. 

Cayley graphs need not be conformally rigid. The smallest example we found is shown on the left in Fig.~\ref{fig:prism}. This unweighted graph has $\lambda_2 = 2$ and $\lambda_6 =5$.
As it turns out, both of these eigenvalues can be improved by changing the edge weights: by increasing the weights on the three edges connecting the two triangles (while decreasing the weight on the other edges to maintain the normalization), we can increase $\lambda_2$ and attain $\lambda_2(w) = 18/7 > 2$. Likewise, by setting the three edge weights connecting the two triangles to 0 (and increasing the remaining weights uniformly) we can decrease the largest eigenvalue to $\lambda_6(w) = 4.5 < 5$ (and thus, in particular, can achieve values smaller than 5 for small positive weights by continuity of the eigenvalues). On the other hand, the circulant graph shown on the right of Fig.~\ref{fig:prism} is conformally rigid but not edge-transitive. Its conformal rigidity is implied by the following Theorem.

\begin{figure}[h!]
        \begin{tikzpicture}
      \node at (0,0) {  \includegraphics[width=0.25\textwidth]{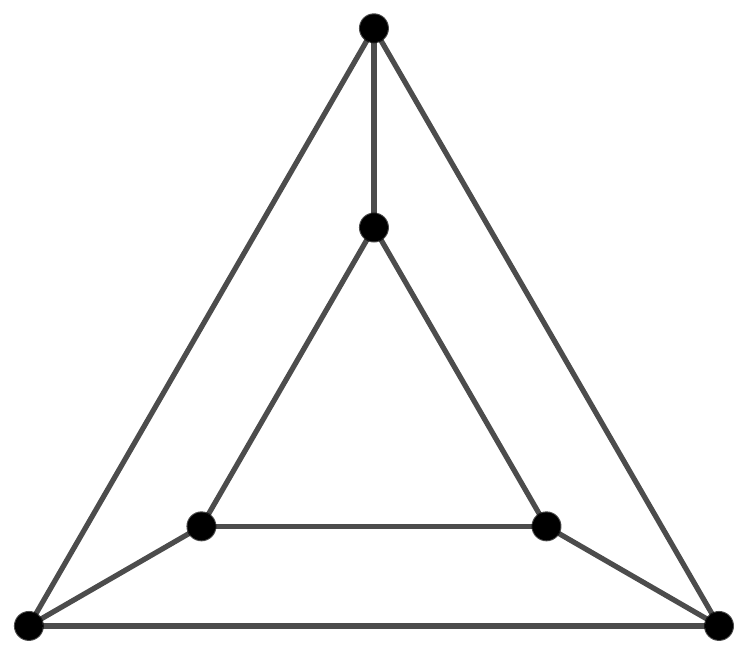}};
       \node at (6,0) {  \includegraphics[width=0.25\textwidth]{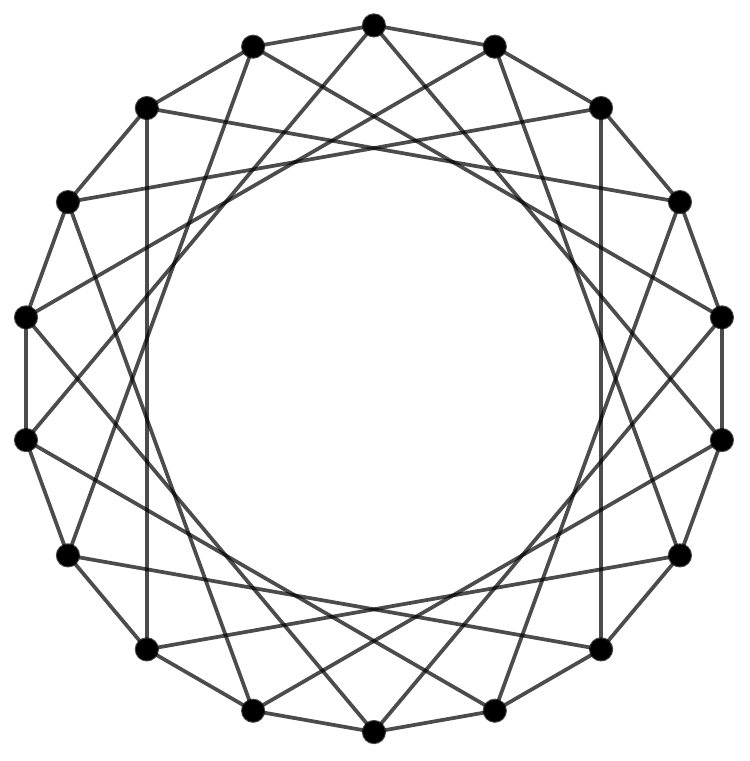}};
       \end{tikzpicture}
        \caption{Left: The triangular prism graph is  Cayley and not conformally rigid. Right: a Cayley graph on $\mathbb{Z}_{18}$ (generated by $S=\left\{-5,-1,1,5\right\}$) that is conformally rigid and \textit{not} edge-transitive.}
        \label{fig:prism}
    \end{figure}
\begin{theorem} \label{thm:Cayley}
    Let $\Gamma(G,S)$ be a finite Cayley graph. If there exist eigenvectors $\phi_2, \phi_n$ corresponding to $\lambda_2, \lambda_n$ such that for $\phi \in \left\{ \phi_2, \phi_n \right\}$ and $s \in S$
    \begin{align} \label{eq:Cayley sum property}
     \sum_{g \in G}  \phi(g) \phi(g \circ s) \qquad \mbox{is independent of}~s,
    \end{align}
    then $\Gamma(G,S)$ is conformally rigid.
\end{theorem}
\textbf{Comments.} Several comments are in order.
\begin{enumerate}
    \item The condition in \eqref{eq:Cayley sum property} is simple to check for any given eigenvector $\phi$: one has to compute $|S|$ different sums and check whether the values all coincide. This is illustrated below where we use Theorem~\ref{thm:Cayley} to prove that the circulant shown in Fig.~\ref{fig:prism} is conformally rigid. 
    \item Recall that $S$ is symmetric ($s \in S \implies -s \in S$). By  changing variables one sees that
        $\sum_{g \in \Gamma}  \phi(g) \phi(g \circ s) = \sum_{g \in \Gamma}  \phi(g) \phi(g\circ (-s)),$
       and hence it suffices to check $|S|/2$ sums in \eqref{eq:Cayley sum property}.
    \item The difficulty in applying Theorem~\ref{thm:Cayley} is that one has to find suitable eigenvectors. What we have observed in practice is that when the eigenspace corresponding to, say, $\lambda_2$, has multiplicity larger than 1, then the property will typically only be true for a single element in the space.
    \item The result is sufficient. It is not clear to us whether it is necessary. In small-scale numerical experiments (carried out exclusively in the setting of circulants) we found that every time a circulant turned out to be conformally rigid, there indeed seemed to be a suitable choice of eigenvectors $\phi_2$ and $\phi_n$ which made the result applicable. It would be interesting to understand this better. 
\end{enumerate}

\begin{example} \label{ex:C18}
    For the circulant graph on the right in Fig.~\ref{fig:prism}, 
     $V = \ZZ_{18}=\{0,1,2,\ldots, 17\}$ and $S = \{1,-1,5,-5\}$ 
    we choose
    $$\lambda_2 = 2, \,\,\,\phi_2 = (-1, -1, 0, 1, 1, 0, -1, -1, 0, 1, 1, 0, -1, -1, 0, 1, 1, 0)$$
    and
    $$\lambda_{18} = 8, \,\,\,\phi_{18} = (-1, 1, -1, 1, -1, 1, -1, 1, -1, 1, -1, 1, -1, 1, -1, 1, -1, 1).$$
    A simple computation then shows that for each $s \in \{1,-1,5,-5\}$,
    $$\sum_{i \in \ZZ_{18}}  \phi_2(i) \phi_2(i+s) = 6 \,\,\,\textup{ and } 
    \sum_{i \in \ZZ_{18}}  \phi_{18}(i) \phi_{18}(i+s) = -18.$$
    Theorem~\ref{thm:Cayley} then implies the conformal rigidity of this circulant. 
\qed
\end{example}

\subsection{Circulants}
Denote a circulant with $n$ vertices and generators $S$ by $C_n(S)$. It is typical to list only one of $s$ or $-s$ in $S$, with the assumption that the other one is included. 
The set of vertices of $C_n(S)$ 
is $V = \ZZ_n = \left\{0,1,2,\dots, n-1\right\}$ and its edges are  $(v, v+s)$, as we vary over all $v \in V$ and $s \in S$ (mod $n$). Since circulants are Cayley graphs,  Theorem~\ref{thm:Cayley} applies to them. We start with some examples.

\begin{enumerate}
    \item The cycle graph $C_n$ and the complete graph $K_n$ are circulants. These are edge-transitive and hence conformally rigid by Proposition~\ref{prop:edge-transitive}. 
    \item Paley graphs of prime order are circulants, edge-transitive and thus also conformally rigid.
    \item Cocktail Party graphs are circulants, distance-transitive and hence,
    edge-transitive, and therefore conformally rigid.
    \item The smallest circulant on two generators which is \textit{not} edge-transitive but nonetheless conformally rigid is $C_{18}(\left\{1,5\right\})$ (see Fig.~\ref{fig:prism}).
\end{enumerate}

Circulants are rich and rather nontrivial. In particular, understanding when a circulant is edge-transitive appears to be a nontrivial problem, no definite characterization appears to be known. There are infinitely many circulants that are \textit{not} conformally rigid.

\begin{proposition} \label{prop:circ}
    For all $n \geq 7$, the circulant $C_n(\left\{1,2\right\})$ is \emph{not} conformally rigid.
\end{proposition}

The restriction $n \geq 7$ is necessary, $C_5(\left\{1,2\right\}) = K_5$ is conformally rigid and so is $C_6(\left\{1,2\right\})$ (also known as the octahedral graph). The list of examples in (1)-(4) contains several infinite families of conformally rigid circulants, however, all of them are edge-transitive. It is a natural question whether there are infinitely many conformally rigid circulants that are not edge-transitive. We were unable to answer this question. Numerically, they appear to be fairly frequent. As an example, the complete list of conformally rigid circulants on $n=21$ vertices with 2 generators is listed below (with graphs that are not edge transitive underlined):

\begin{figure}[h!]
\begin{center}
$$ \underline{C_{21}(\left\{1,6\right\})}, C_{21}(\left\{1,8\right\}), C_{21}(\left\{2,5\right\}), \underline{C_{21}(\left\{2,9\right\})}, \underline{C_{21}(\left\{3,4\right\})}, \underline{C_{21}(\left\{3,10\right\})}$$
$$C_{21}(\left\{4,10\right\}), \underline{C_{21}(\left\{5,9\right\})}, \underline{C_{21}(\left\{6,8\right\})}, C_{21}(\left\{10,11\right\}).$$
\end{center}
\caption{Conformally rigid circulants on $n=21$ vertices with two generators. Underlined circulants are not edge-transitive.}
\end{figure}

\textbf{Questions.}
We conclude with a number of questions.
\begin{enumerate}
    \item  Is it possible to characterize conformally rigid circulants $C_n(\left\{a,b\right\})$ in terms of $n, a, b$? While there seem to be a great many, we were unable to detect a simple pattern or even identify a single infinite family.
    \item Somewhat unrelated, but of independent interest in this context: is it possible to characterize when $C_n(\left\{a,b\right\})$ is edge-transitive in terms of $n, a, b$? Some partial results can be found in \cite{circ0, circ1, circ2, circ3, circ4, zhang2}.
    \item Theorem~\ref{thm:Cayley} provides a rigorous criterion that can be used to prove conformal rigidity of Cayley graphs and, in particular, circulants. Does the Theorem identify all conformally rigid circulants? Numerically this seems to be true 
    for small conformally rigid circulants.
    \item Are there infinitely many conformally rigid circulants $C_n(\left\{a,b\right\})$ that are \textit{not} edge-transitive? The smallest one appears to be
     $C_{18}(\left\{1,5\right\})$ described in Example~\ref{ex:C18}, but there appear to be many more after that.
     \item All these questions would be also interesting in the context of circulants with three or more generators.
\end{enumerate}

\subsection{Sporadic outliers}
\label{subsec:isolated}
At this point, we have identified a number of reasons why a graph may be conformally rigid; most of the examples can be identified as belonging to a number of different groups, these being
\begin{enumerate}
    \item edge-transitive graphs,
    \item distance-regular graphs, in particular, strongly-regular graphs, 
    \item and certain Cayley graphs, in particular, certain circulant graphs. 
\end{enumerate}

The purpose of this section is to describe some isolated sporadic examples of graphs that do not belong into any of these groups.

\begin{theorem} The following graphs are conformally rigid, not edge-transitive and not distance-regular:
    \begin{enumerate} \label{thm:exc}
        \item the Hoffman graph on 16 vertices (see Fig. \ref{fig:conformally rigid})
        \item the crossing number graph 6B on 20 vertices \cite{crossing} (see Fig. \ref{fig:conformally rigid})
         \item the Haar graph 565 on 20 vertices (see Fig. \ref{fig:Haar})
         \item the distance-2 graph of the Klein graph on 24 vertices (see Fig. \ref{fig:out})
        \item the $(7,1)-$bipartite $(0,2)-$graph on $n=48$ vertices (see Fig. \ref{fig:out})
        \item the $(7,2)-$bipartite $(0,2)-$graph on $n=48$ vertices
      \item the $(20,8)$ accordion graph \cite{accordion} on 40 vertices (see Fig. \ref{fig:out})
      \item the non-Cayley vertex-transitive graph (24,23) on $n=24$ vertices.
    \end{enumerate}
\end{theorem}

\begin{center}
    \begin{figure}[h!]
    \begin{tikzpicture}
            \node at (-4,0) {\includegraphics[width=0.28\textwidth]{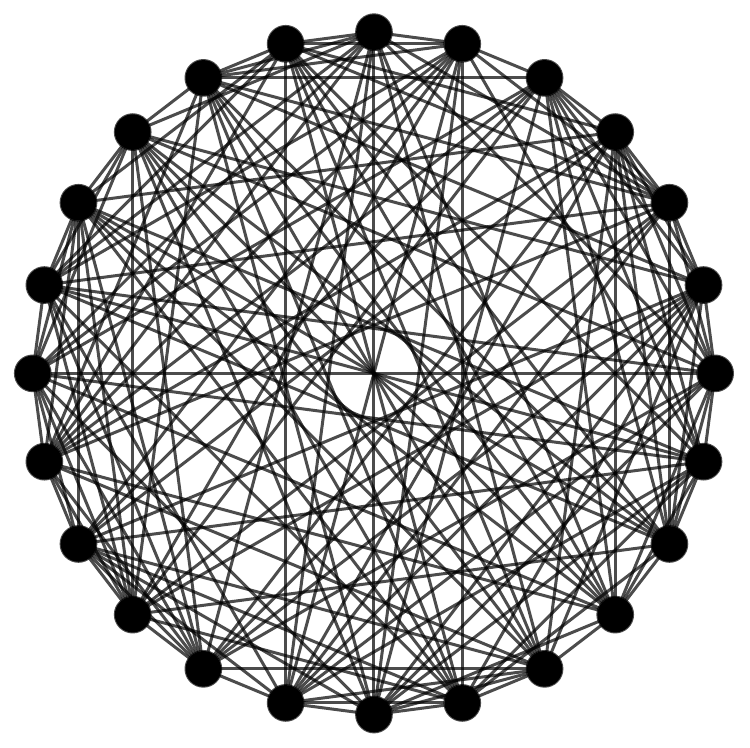}};
        \node at (0,0) {\includegraphics[width=0.28\textwidth]{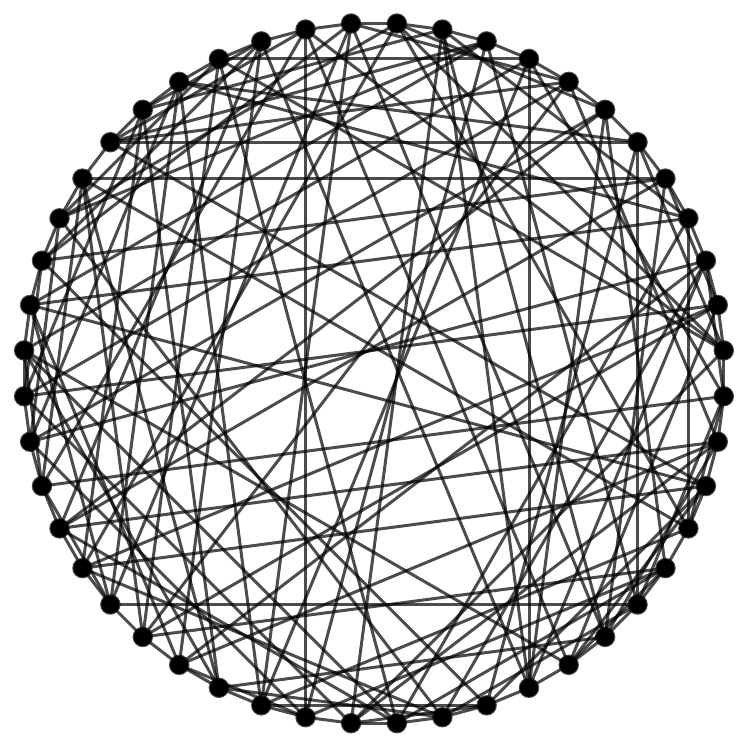}};
        \node at (4,0) {\includegraphics[width=0.28 \textwidth]{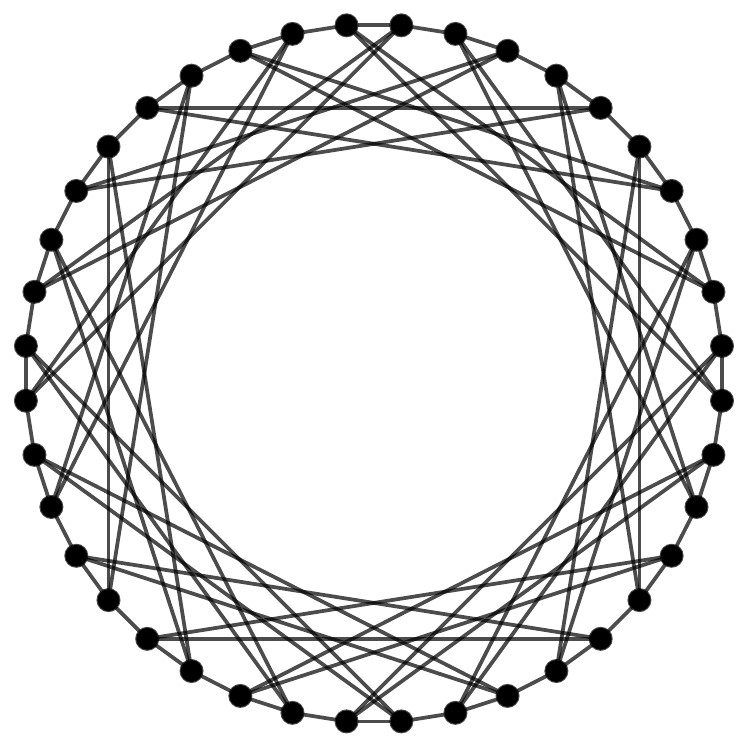}};
    \end{tikzpicture}
    \caption{Sporadic outliers: conformally rigid graphs that are neither edge-transitive nor distance-regular. Left: distance-2 graph of the Klein graph. Middle: a bipartite $(0,2)-$graph on $n=48$ vertices. Right: the $(20,8)$ accordion graph \cite{accordion}.}
    \label{fig:out}
    \end{figure}
\end{center}
\vspace{-10pt}

The graphs in this list are mostly well-known; the less well known examples (3), (5), (6), (8) are implemented in Mathematica 13.3 as ("Haar", 565), ("ZeroTwoBipartite", {7, 1}),("ZeroTwoBipartite", {7, 2}) and ({"NoncayleyTransitive", {24, 23}). The graphs in Theorem \ref{thm:exc} were all proven to be conformally rigid by identification of a suitable dual semidefinite programming certificate (see \S~\ref{sec:sdp certificates}). However, we also found one other example, the non-Cayley vertex-transitive graph (20,10) on $n=20$ vertices, that is not edge-transitive, not distance-regular, but very likely conformally rigid (beyond any reasonable numerical doubt); however, we were unable to find a certificate that would rigorously prove its conformal rigidity.

\section{Semidefinite programming certificates for conformal rigidity} \label{sec:SDP}
\subsection{Introduction.} The conformal rigidity of a graph can be checked and certified 
using {\em semidefinite programming} which is a  branch of convex optimization \cite{boyd-vandenberghe}. This connection provides both 
theoretical and computational tools as we will see in \S~\ref{sec:embeddings} and \S~\ref{sec:sdp certificates}.  We first explain the necessary background.

Assume that $G=([n],E)$ is an undirected connected graph with positive edge weights $w_{ij}$.
Let $\mathcal{S}_+^n$ be the cone of $n \times n$ real symmetric psd matrices. We use $X \succeq 0$ to denote that $X \in \mathcal{S}^n_+$. For  
$A,B \in \mathcal{S}^n_+$,  $A \succeq B$ stands for $A-B \succeq 0$.

Recall from \S~\ref{subsec:opt} that the problems of maximizing $\lambda_2(w)$ and  minimizing $\lambda_n(w)$ with the sum of edge weights being bounded, have been studied in a number of papers such as \cite{fiedler1990trees, goering-helmberg-wappler,goering-helmberg-reiss,sun-boyd-xiao-diaconis}. We follow the exposition in \cite{sun-boyd-xiao-diaconis} where one observes
\begin{align} \label{eq:prob l2}
    \lambda_2^\ast :=  \textup{ max } \left\{ \lambda_2(w) \,:\, \sum w_{ij} = |E|, \,\,
 w \geq 0 \right\}.
\end{align}
is a convex optimization problem 
that is equivalent to 
\begin{align} \label{eq:prob l2 switch}
    p^\ast:=  \textup{ min}  \left\{ \sum_{ij \in E} w_{ij} \,:\, 
 \lambda_2(w) \geq 1, \,\,  w \geq 0 \right\}.
\end{align}

Problem~\eqref{eq:prob l2} has an optimal solution since its objective function is continuous and its feasible region is compact. Therefore, \eqref{eq:prob l2 switch} also has an optimal solution. If $w^*$ is optimal for \eqref{eq:prob l2 switch}, then $\lambda_2(w^*) = 1$ and hence there is no harm in optimizing over $\lambda_2(w) \geq 1$.
The two problems are off by a scaling; $\lambda_2^* = |E|/ p^*$ and $w^*$ is optimal for \eqref{eq:prob l2 switch} if and only if $(|E| w^*)/p^*$ is optimal for \eqref{eq:prob l2}.
Using $I$ to denote the identity matrix and $J$ to denote the matrix filled with 1's, the constraint $\lambda_2(w) \geq 1$ can be modeled by $L_w \succeq I - (1/n) J$ (the eigenvalues of $I - (1/n) J$ are 
$1,1,\ldots,1,0$) and \eqref{eq:prob l2 switch} is  the semidefinite program (SDP):
\begin{align} \label{eq:prob l2 sdp}
\begin{split}
    p^\ast:= & \textup{ min}  \sum_{ij \in E} w_{ij}\\
    \textup{s.t. }  & L_w \succeq I - \frac{1}{n} J\\
    & w \geq 0.
\end{split}    
\end{align}
SDPs are optimization problems 
in the space of symmetric matrices where we are optimizing a linear function over a (convex) region described by  linear matrix inequalities, or equivalently, an affine slice of the cone of 
psd matrices, see \cite{boyd-vandenberghe}. Calling \eqref{eq:prob l2 sdp} the primal SDP, its dual SDP is derived in \cite{sun-boyd-xiao-diaconis} as
\begin{align} \label{eq:prob l2 sdp dual}
\begin{split}
    d^\ast:= & \textup{ max }  \textup{Trace } X\\
    \textup{s.t. }  & X_{ii} + X_{jj} - 2 X_{ij} \leq 1 \,\,\forall ij \in E\\
    & \ones^\top X \ones = 0, X \succeq 0.
\end{split}    
\end{align}

{\em Strong duality} holds between the primal and dual SDPs and so $p^* = d^*$. A pair of primal/dual feasible solutions $w^*$ and $X^*$ are optimal for their respective problems if and only if they satisfy the {\em complementary slackness conditions}
\begin{align}
    (1-(X_{ii}^* + X_{jj}^*-2X_{ij}^*)) w_{ij}^* = 0  \,\,\forall ij \in E \label{eq:cs 1}\\
    L_{w^*}X^* =  X^*  \label{eq:cs 2}
\end{align}

Analogous to the above, the problem of minimizing $\lambda_n(w)$ can be written as 
\begin{align} \label{eq:prob ln}
    \lambda_n^\ast :=  \textup{ min } \left\{ \lambda_n(w) \,:\, 
 \sum w_{ij} = |E|, \,\, 
    w \geq 0 \right\}.
\end{align}
which is equivalent to 
\begin{align} \label{eq:prob ln switch}
    q^\ast:=  \textup{ max} \left\{  \sum_{ij \in E} w_{ij} \,:\, 
    \lambda_n(w) \leq 1, \,\, 
 w \geq 0 \right\}.
\end{align}

As before, if $w^*$ is optimal for 
\eqref{eq:prob ln switch} then $\lambda_n(w^*) = 1$, $\lambda_n^* = |E|/q^*$ and 
$w^*$ is optimal for \eqref{eq:prob ln switch} if and only if 
$(|E| w^*)/q^*$ is optimal for \eqref{eq:prob ln}. Problem~\eqref{eq:prob ln switch} is the SDP:

\begin{align} \label{eq:prob ln sdp}
\begin{split}
    q^\ast:= & \textup{ max}  \sum_{ij \in E} w_{ij}\\
    & \textup{s.t. } L_w \preceq I - \frac{1}{n} J\\
    & w \geq 0
\end{split}    
\end{align}
whose dual is 
% -- actual dual 
% \begin{align} \label{eq:prob ln sdp dual}
% \begin{split}
%     r^\ast:= & \textup{ min }  \textup{Trace } (I-\frac{1}{n}J) Y\\
%     & \textup{s.t. }  Y_{ii} + Y_{jj} - 2 Y_{ij} \geq 1 \,\,\forall ij \in E\\
%     &  Y \succeq 0.
% \end{split}    
% \end{align}
\begin{align} \label{eq:prob ln sdp dual}
\begin{split}
    r^\ast:= & \textup{ min }  \textup{Trace } Y\\
    & \textup{s.t. }  Y_{ii} + Y_{jj} - 2 Y_{ij} \geq 1 \,\,\forall ij \in E\\
    &  \ones^\top Y \ones = 0, \,\,Y \succeq 0.
\end{split}    
\end{align}
Again, strong duality holds so that $q^* = r^*$, and a pair of primal/dual feasible solutions $(w^*,Y^*)$ are optimal for their respective problems if and only if they satisfy complementary slackness:
\begin{align}
    (Y^*_{ii} + Y^*_{jj} - 2 Y^*_{ij} - 1)w^*_{ij} = 0 \,\,\forall ij \in E \label{eq:cs 3}\\
    L_{w^*}Y^*= Y^* \label{eq:cs 4}
\end{align}

We can now test for conformal rigidity via the above optimization problems.
\begin{observation} \label{obs:opt cert}
    A graph $G$ is conformally rigid if and only if 
\begin{enumerate}
        \item $\ones/\lambda_2(\ones)$ is an optimal solution to \eqref{eq:prob l2 switch} (or \eqref{eq:prob l2 sdp}) with  $p^* = |E|/\lambda_2(\ones)$, and
        \item $\ones/\lambda_n(\ones)$ is an optimal solution to \eqref{eq:prob ln switch} (or \eqref{eq:prob ln sdp}) with $q^* = |E|/\lambda_n(\ones)$.
\end{enumerate}
% \item \begin{enumerate}
%       \item The SDP \eqref{eq:prob l2 sdp dual} has an optimal solution $X^*$ with 
%       $\textup{Tr}(X^*) = |E|/\lambda_2(\ones)$, 
%       \centerline{and}
%       \item the SDP \eqref{eq:prob ln sdp dual} has an optimal solution $Y^*$ with 
%       $\textup{Tr}(Y^*) = |E|/\lambda_n(\ones)$.
%      \end{enumerate}
%\end{enumerate}
\end{observation}

Under strong duality, the pair of primal and dual SDPs will have the same optimal value, which allows the following re-phrasing:
\begin{observation} 
    A graph $G$ is conformally rigid if and only if 
\begin{enumerate}
      \item The SDP \eqref{eq:prob l2 sdp dual} has an optimal solution $X^*$ with  $\textup{Trace}(X^*) = |E|/\lambda_2(\ones)$, and 
      \item the SDP \eqref{eq:prob ln sdp dual} has an optimal solution $Y^*$ with 
       $\textup{Trace }(Y^*) = |E|/\lambda_n(\ones)$.
\end{enumerate}
\end{observation}

Usually, these observations cannot be used directly to certify conformal rigidity since SDPs are solved using numerical algorithms making it difficult to know precisely what their optimal values and  solutions are. In \S~\ref{sec:sdp certificates} we illustrate several methods to overcome these difficulties. 
A more computationally friendly certificate for conformal rigidity can be obtained from the complementary slackness conditions. Indeed, 
when $G$ is conformally rigid, the optimal weights are $\ones/\lambda_2(\ones)$ and $\ones/\lambda_n(\ones)$  and so complementary slackness implies the following.

\begin{proposition} \label{prop:certificates}
A graph $G$ is conformally rigid if and only if 
there is a feasible solution $X$ for 
\eqref{eq:prob l2 sdp dual} and a feasible solution $Y$ for \eqref{eq:prob ln sdp dual} such that 
\begin{align} 
    X_{ii} + X_{jj}-2X_{ij} = 1  \quad \,\,\forall ij \in E, \,\,\,
    LX= \lambda_2(\ones) X \label{eq:cs x}\\
    Y_{ii} + Y_{jj}-2Y_{ij} = 1  \quad \,\,\forall ij \in E, \,\,\,
    LY= \lambda_n(\ones) Y. \label{eq:cs y}
\end{align}
\end{proposition}

The condition $ LX= \lambda_2(\ones) X$
is saying that the columns of $X$ are in the eigenspace of $L$ for  $\lambda_2(\ones)$. 
Therefore, the rank of $X$ is bounded above by the multiplicity of $\lambda_2(\ones)$ as an eigenvalue of $L$. 
Analogously, the columns of $Y$ must come from the eigenspace of 
$L$ for $\lambda_n(\ones)$ and so the rank of $Y$ cannot exceed the multiplicity of $\lambda_n(\ones)$. In particular, if $\lambda_2(\ones)$ (respectively, $\lambda_n(\ones)$) has multiplicity $1$, then $X$ (respectively, $Y$) can be taken to be the rank one matrix 
\begin{align}\label{obs:rank one certificate}
p^* uu^\top \textup{} (\textup{ respectively, } q^* uu^\top) 
\end{align}
where $u$ is a normalized eigenvector of $\lambda_2(\ones)$ 
    (respectively, $\lambda_n(\ones)$).
These observations were noted in \cite{sun-boyd-xiao-diaconis}. A very short explanation why the notion of conformal rigidity simplifies when the eigenspace associated to $\lambda_2$ (or $\lambda_n$) has multiplicity 1 is as follows: in that case, we have
$$ \lambda_2(\ones) = \min_{\sum_{v \in V} f(v) = 0} \frac{\sum_{(i,j) \in E}  (f(i) - f(j))^2}{ \sum_{v \in V} f(v)^2}$$
We note that introducing weights $w_{ij}$ (normalized to $\sum_{(i,j) \in E}w_{ij} = |E|$) and then continuously varying the weights leads to a continuously perturbation of the matrix. If $\lambda_2$ has multiplicity 1, then both the eigenvalue $\lambda_2(w)$ and the associated eigenvector depend continuously on the change in the weights (and, in particular, do not change much for small changes in the weight).
Suppose now that, for the eigenvector in the unweighted case, not all terms in the sum $(f(i) - f(j))^2$ are of the same size: then we could consider a small perturbation of the weights, putting more emphasis on larger summands and less on smaller ones and arrive at a weight for which $\lambda_2(w) > \lambda_2(\ones)$ which contradicts conformal rigidity.\\

% \begin{corollary} \label{cor:rank one certificate}
%     If $\lambda_2(\ones)$ (respectively, $\lambda_n(\ones)$) has multiplicity $1$, then $X$ (respectively, $Y$) can be taken to be the rank one matrix $p^* uu^\top$  (respectively, $q^* uu^\top$) where  and $u$ is a normalized eigenvector of $\lambda_2(\ones)$ 
%     (respectively, $\lambda_n(\ones)$).
% \end{corollary}

We now do a small example to illustrate the above ideas, one in which all claims can be checked by hand. 

\begin{example}
    Let $G$ be the $4$-cycle with edges $12,23,34,14$ which is edge-transitive and hence conformally rigid by Proposition~\ref{prop:edge-transitive}.
    The Laplacian of $G$ has eigenvalues: $0,2,2,4$ with eigenvectors:
    $$(1,1,1,1), \underbrace{(-1,0,1,0),(0,-1,0,1)}_{\lambda = 2},\underbrace{(-1,1,-1,1)}_{\lambda=4}.$$
    The weights 
    $$w^* = (1/2,1/2,1/2,1/2), \,\,\,\textup{ and } 
    \,\,\,w'^* = (1/4,1/4,1/4,1/4)$$
    are optimal solutions to \eqref{eq:prob l2 switch} and \eqref{eq:prob ln switch}. To verify these claims, we construct certificates $X,Y$ as in Proposition~\ref{prop:certificates}.
    Take

$$ X = \frac{1}{2}\begin{pmatrix} 
1 & 0 & -1 & 0 \\ 
0 & 1 & 0 & -1\\
-1 & 0 & 1 & 0 \\
0 & -1 & 0 & 1
\end{pmatrix}, \,\,\,\,
Y = \frac{1}{4} \begin{pmatrix}
1 & -1 & 1 & -1 \\
-1 & 1 & -1 & 1 \\
1 & -1 & 1 & -1\\
-1 & 1 & -1 & 1
\end{pmatrix}.$$
Check that $X$ is a feasible solution of \eqref{eq:prob l2 sdp dual} that satisfies
\eqref{eq:cs x}; the eigenvalues of $X$ are $0,0,1,1$ which shows that $X \succeq 0$. The analogous checks can  be made for $Y$, and $(X,Y)$ is a pair of SDP certificates for the conformal rigidity of $G$.

Note that $\textup{Trace } X = 2 = \ones^\top w^*$ and $\textup{Trace } Y = 1 = \ones^\top w'^*$. Since the multiplicity of $\lambda_2$ is $2$, $\rank(X) \leq 2$. In fact, $\rank(X) = 2$.  Since the multiplicity of $\lambda_4 = 4$ is $1$, we expect $Y$ to be a scaling of $uu^\top$ where $u=(-1,1,-1,1)$ is the eigenvector of $4$ shown above, and this is indeed the case. 
 \qed    
\end{example}

\subsection{Symmetry reduction} 
The symmetries of $G$ allow the SDPs 
\eqref{eq:prob l2 sdp} and \eqref{eq:prob ln sdp} to be reduced to have less variables, allowing them to be more easily solvable and helping to understand their structure. We briefly explain this well-known technique in our setting and use it to prove the conformal rigidity of some graphs where direct methods were not successful. 
This exposition closely follows that in \cite{boyd-diaconis-parrilo-xiao}.

Recall that a permutation $\pi \in \textup{Aut}(G)$ acts on vertices, edges and weights of $G$ via $i \mapsto \pi(i), \,\,ij \mapsto \pi(i) \pi(j), \,\, w_{ij} \mapsto w_{\pi(i)\pi(j)}$. Let $P$ denote the 
$n \times n$ permutation matrix corresponding to $\pi$ with $P_{ij} = 1$ if $\pi(j) = i$ and $0$ otherwise. If we denote the weight obtained by the action of $\pi$ on $w$ by $\pi \cdot w$, then  
$L_{\pi \cdot w} = P L_w P^\top$. 

Now consider the SDP \eqref{eq:prob l2 sdp} and its feasible region 
\begin{align} \label{eq:feasible region}
    \mathcal{F} := \{ w \in \RR^E \,:\, L_w \succeq I - \frac{1}{n} J, \,\, w \geq 0 \}. 
\end{align}
If $w \in \mathcal{F}$ and $\pi \in \textup{Aut}(G)$, then 
$\pi \cdot w \geq 0$ and 
$$ L_{\pi \cdot w} = P L_w P^\top \succeq P(I - \frac{1}{n} J)P^\top = I - \frac{1}{n} J.$$
Therefore, $\mathcal{F}$ is invariant under the action of $\textup{Aut}(G)$.

By the invariance of $\mathcal{F}$, if $w^*$ is an optimal solution of \eqref{eq:prob l2 sdp}, then $\pi \cdot w^* \in \mathcal{F}$ for all $\pi \in \textup{Aut}(G)$ and therefore, by the convexity of $\mathcal{F}$, we also have that
\begin{align} \label{eq:average}
\overline{w^\ast} := \frac{1}{|\textup{Aut}(G)|} \sum_{\pi} \pi \cdot w^\ast \in \mathcal{F}.
\end{align}
Further, $\overline{w^\ast}$ is also an optimal solution to \eqref{eq:prob l2 sdp} since $\ones^\top (\pi \cdot w^*) = 
\ones^\top w^* = p^*$. This optimal solution $\overline{w^\ast}$ is fixed under the action of $\textup{Aut}(G)$. Thus we may 
restrict the optimization in \eqref{eq:prob l2 sdp} to the fixed point subset, $\overline{\mathcal{F}} \subset \mathcal{F}$.

Let $O(i), i=1,\ldots,k$ be the orbits of edges in $E$ under the action of $\textup{Aut}(G)$. Any fixed point $w \in \overline{\mathcal{F}}$ 
must have the same value in all coordinates lying in an edge orbit 
$O(i)$. Let $w_i$ be the (variable) weight on the edges in $O(i)$. Then \eqref{eq:prob l2 sdp} is equivalent to the symmetry reduced SDP: 
\begin{align} \label{eq:prob l2 sdp symm}
\begin{split}
    p^\ast:= & \textup{ min}  \sum_{i=1}^k |O(i)| w_i\\
    \textup{s.t. }  & \sum_{i=1}^k w_i L_i \succeq I - \frac{1}{n} J\\
    & w_i \geq 0 \,\,\forall i
\end{split}    
\end{align}
where $L_i$ is the Laplacian of the graph $G_i = ([n],O(i))$.  The symmetry reduction for the SDP \eqref{eq:prob ln sdp} is analogous and yields the SDP: 
\begin{align} \label{eq:prob ln sdp symm}
\begin{split}
    q^\ast:= & \textup{ max}  \sum_{i=1}^k |O(i)| w_i\\
    \textup{s.t. }  & \sum_{i=1}^k w_i L_i \preceq I - \frac{1}{n} J\\
    & w_i \geq 0 \,\,\forall i
\end{split}    
\end{align}

The dual SDPs \eqref{eq:prob l2 sdp dual} and \eqref{eq:prob ln sdp dual} can also be symmetry reduced in a similar way.

\section{Certificates and Spectral Embeddings} \label{sec:embeddings}
It was noted in \cite{goering-helmberg-wappler} and \cite{sun-boyd-xiao-diaconis} that the dual certificates $X$ and $Y$ in Proposition~\ref{prop:certificates} provide 
embeddings of $G$. In the context of conformal rigidity, this 
interpretation can be developed further, and it will be our main
tool to prove that all distance-regular graphs are conformally rigid.

\begin{definition} \label{def:embedding}
Let $G=([n],E)$ be a connected graph and let $\lambda >0$  be an eigenvalue of its Laplacian 
$L$, of multiplicity $m$. Let $\mathcal{E}_\lambda \cong \RR^{m}$ be the eigenspace of $\lambda$ 
and $P \in \RR^{n \times k}$ a matrix whose 
columns lie in $\mathcal{E}_\lambda$.  The collection of vectors 
$\mathcal{P} = \{ 
p_i, \,\,i=1, \ldots, n\} \subset \RR^{k}$, such that $p_1^\top, \ldots p_n^\top$ are the rows of $P$,  
is called an {\em embedding} of $G$ on $\mathcal{E}_\lambda$. 
\end{definition}

Indeed, the configuration $\mathcal{P}$ provides a realization/embedding of $G$ in $\RR^{k}$ by assigning the vector $p_i$ to vertex $i$, and connecting  $p_i$ and $p_j$ by an edge for all $ij \in E$. The embedding is {\em centered} in the sense that $\sum p_i = 0$ since the columns of $P$ are 
orthogonal to $\ones$.  
Spectral embeddings as above were defined by Hall \cite{hall} and several variants have been studied. Typically, there are 
more requirements on the matrix $P$ such as the columns of $P$ should form an orthonormal basis of $\mathcal{E}_\lambda$. When $P$ is chosen to 
have this extra property, we will write $U$ in the place of $P$ and $\mathcal{U}$ in the place of $\mathcal{P}$. The embedding 
is {\em spherical} if $\|p_i\| = \alpha > 0$ for all $i=1, \ldots, n$. We will see embeddings with these additional properties for distance-regular graphs. 

\begin{definition} \label{def:edge iso embedding}
    An embedding $\mathcal{P}$ of $G$ on $\mathcal{E}_\lambda$ is said to be 
    {\em edge-isometric} if there exists a constant $c > 0$ such that 
    $\|p_i - p_j \|=c$ for all $ij \in E$. 
\end{definition}

The following is an adaptation of an observation in 
\cite{goering-helmberg-wappler} and \cite{sun-boyd-xiao-diaconis} in the context of conformal rigidity. 

\begin{proposition} \label{prop:embedding certificates}
A connected graph $G=([n],E)$ is conformally rigid if and 
only if $G$ has an edge-isometric embedding $\mathcal{P}$ on $\mathcal{E}_{\lambda_2(\ones)}$ 
and an edge-isometric embedding $\mathcal{Q}$ on $\mathcal{E}_{\lambda_n(\ones)}$.
\end{proposition}

\begin{proof}
Suppose $G$ is conformally rigid. Then there exists matrices $X$ and $Y$ satisfying the conditions of Proposition~\ref{prop:certificates}.
Since $X$ and $Y$ are psd, they admit a Gram decomposition $X=PP^\top$ and 
$Y = QQ^\top$. The columns of $P$ lie in $\mathcal{E}_{\lambda_2(\ones)}$ since the columns of $X$ do. Similarly, the columns of $Q$ lie in $\mathcal{E}_{\lambda_n(\ones)}$ as the columns of $Y$ do. Therefore, $\mathcal{P}$ and $\mathcal{Q}$, consisting of the rows of $P$ and $Q$ respectively, are embeddings of $G$ on $\mathcal{E}_{\lambda_2(\ones)}$  and $\mathcal{E}_{\lambda_n(\ones)}$ respectively. The condition 
$$1=X_{ii} + X_{jj} - 2X_{ij} = \|p_i - p_j\|^2 \,\,\,\forall ij \in E$$
ensures that $\mathcal{P}$ is edge-isometric. Similarly for $\mathcal{Q}$.

Conversely, suppose we have embeddings as in the statement of the Proposition. 
Let $\mathcal{P}$ be the edge-isometric embedding on $\mathcal{E}_{\lambda_2(\ones)}$. Suppose $\|p_i - p_j\| = c >0$ for all $ij \in E$ and the columns of $P$ lie in $\mathcal{E}_{\lambda_2(\ones)}$. Define $X := c^{-2} PP^\top \in \mathcal{S}^n_+$.
Then $\ones^\top X \ones = 0$ since the columns of $P$ are eigenvectors of 
$\lambda_2(\ones)$ and are hence orthogonal to $\ones$. Since $X$ and $P$ have the same column space, $L X = \lambda_2(\ones) X$. 
If $ij \in E$, then 
$$X_{ii} + X_{jj} - 2X_{ij} = \frac{1}{c^2} (\|p_i\|^2 + \|p_j\|^2 - 2 p_i^\top p_j ) = \frac{1}{c^2}( \|p_i - p_j \|^2) =1.$$
Therefore, $X$ satisfies the conditions in Proposition~\ref{prop:certificates}. By the same argument we get a 
certificate $Y$ needed in Proposition~\ref{prop:certificates} from the 
edge-isometric embedding of $G$ on $\mathcal{E}_{\lambda_n(\ones)}$.
Together, $X$ and $Y$ certify that $G$ is conformally rigid. 
\end{proof}

\begin{remark} \label{rem:scaling X Y}
The certificate $X = c^{-2}PP^\top$ in the  proof of Proposition~\ref{prop:embedding certificates} is an optimal solution to \eqref{eq:prob l2 sdp dual}. Therefore, $\textup{Trace } X =  c^{-2} \textup{Trace }PP^\top = |E|/\lambda_2(\ones)$ which implies that $c^2 = (\lambda_2(\ones) \sum \|p_i\|^2) /|E|$. 

Similarly, if $\mathcal{Q}=\{q_1, \ldots, q_n\}$ is the edge-isometric embedding on $\mathcal{E}_{\lambda_n(\ones)}$ that gave rise to the certificate $Y$, then $\textup{Trace } Y = (c')^{-2} \sum \|q_i\|^2 = |E|/\lambda_n(\ones)$ which implies that $c'^2 = 
(\lambda_n(\ones) \sum \|q_i\|^2) /|E|$. 
\end{remark}

Next, we  define a more ambitious type of embedding towards a characterization of when $G$ and its complement are both 
conformally rigid. 

\begin{definition}\label{def:2 distance embedding}
An embedding $\mathcal{P}$ of $G = ([n],E)$ on $\mathcal{E}_\lambda$ as in Definition~\ref{def:edge iso embedding} is {\em edge-nonedge-isometric} if there are 
two nonzero constants $\alpha, \beta > 0$ such that 
\begin{enumerate}
    \item $\|p_i - p_j \| = \alpha \,\,\,\forall ij \in E$, and  
    \item $\|p_i - p_j \| = \beta \,\,\,\forall ij \not \in E$.
\end{enumerate}
\end{definition}

Conformal rigidity of $G$ does not imply that $G$ has an  embedding on 
$\mathcal{E}_{\lambda_2}$ or $\mathcal{E}_{\lambda_n}$ that is edge-nonedge-isometric.

\begin{example} \label{ex:Hoffman}
    The Hoffman graph on $16$ vertices in Fig~\ref{fig:conformally rigid} is conformally rigid, but its complement is not. Its Laplacian eigenvalues, with multiplicities, are: $$ (\lambda, \textup{mult}(\lambda)): ( 8,1), (6,4), ( 4,6), (2,4), ( 0,1).$$
   The eigenspace $\mathcal{E}_8$ is spanned by the vector 
   $$v^\top = (-1, -1, -1, -1, -1, -1, -1, -1, 1, 1, 1, 1, 1, 1, 1, 1)$$ 
   and therefore, consider the ordered configuration
    $$\mathcal{Q} = \left\{
 -1, -1, -1, -1, -1, -1, -1, -1, 1, 1, 1, 1, 1, 1, 1, 1
\right\}.$$
Once can check that $\|q_i - q_j\| = 2$ for all $ij \in E$, while if 
$ij \not \in E$, then $\|q_i-q_j\| = 0,2$. For instance, 
$(1,9), (1,10) \in E$ and $(1,2),(1,16) \not \in E$.
Therefore, there are two 
different distances between non-adjacent vertices of $G$ in the embedding provided by $\mathcal{Q}$.  
Since the eigenspace $\mathcal{E}_{8}$ is $1$-dimensional, any 
other embedding on this eigenspace must be a scaling of the configuration $\mathcal{Q}$. The property of there being two different distances between non-adjacent vertices is invariant under scaling. 
\qed
\end{example}

% Another example is genralized quadrangle minus spread (2,4),1
% which is CR but it's complement is not. There is a unique Y 
% and it has two different distances for the non-edges.

In the rest of this section, assume that $G = ([n],E)$ is a regular graph 
of valency $d$. Then its complement $G^c = ([n],E^c)$ 
is a $(n-d-1)$-regular graph on $[n]$. Since the adjacency 
matrices are related as $A_{G^c} = J - A_G -I$ where $J$ is the matrix of all ones, 
we have that $L_{G^c} = (n-d)I - J + A_G$.
If  $G$ and $G^c$ are both connected, then $0$ is an eigenvalue of both $L_G$ and $L_{G^c}$, and $\lambda >0$ is an 
eigenvalue of $L_G$ if and only if $n-\lambda$ is an eigenvalue of $L_{G^c}$, and  
both eigenvalues have the same eigenspace which we will write as 
$\mathcal{E}_\lambda = \mathcal{E}^c_{n - \lambda}$.  

Conformal rigidity is not closed under graph complementation.
A simple example is the $6$-cycle $C_6$ which is conformally rigid as it is edge-transitive, but its complement is not, see Example~\ref{ex:C6 and its complement}. 
Proposition~\ref{prop:certificates} provides the following 
certification for the conformal rigidity of both $G$ and $G^c$. 

\begin{corollary} \label{cor:G and Gc cr}
Let $G= ([n],E)$ be a connected regular graph whose complement $G^c$ is also connected. Then $G$ and $G^c$ are both conformally rigid if and only if there exists $X,Y,X',Y' \in \mathcal{S}^n_+$ such that 
    \begin{align}
        X_{ii}+X_{jj}-2X_{ij}=1 \forall ij \in E, \,\,\,
        \textup{cols}(X) \subset \mathcal{E}_{\lambda_2(\ones)},\,\,\,
        \ones^\top X \ones = 0 \label{eq:x}\\
        Y_{ii}+Y_{jj}-2Y_{ij}=1 \forall ij \in E, \,\,\,
        \textup{cols}(Y) \subset \mathcal{E}_{\lambda_n(\ones)}\,\,\,\ones^\top Y \ones = 0 \label{eq:y}\\
        X'_{ii}+X'_{jj}-2X'_{ij}=1 \forall ij \not \in E, \,\,\,\textup{cols}(X') \subset  \mathcal{E}_{\lambda_n(\ones)}, \,\,\,\ones^\top X' \ones = 0 \label{eq:x'}\\
        Y'_{ii}+Y'_{jj}-2Y'_{ij}=1 \forall ij \not \in E, \,\,\,\textup{cols}(Y') \subset \mathcal{E}_{\lambda_2(\ones)}, \,\,\,\ones^\top Y' \ones = 0. \label{eq:y'}
    \end{align}
    \end{corollary}

It would be particularly convenient if we could take $X'=Y$ and $Y'=X$. Here is an embedding version of this wish, that follows from 
Proposition~\ref{prop:embedding certificates}.

\begin{corollary}\label{cor:edge-non-edge-iso implies G & G^c CR }
    Let $G$ be a connected regular graph such that $G^c$ is also connected. 
    If $G$ has an embedding $\mathcal{P}$ on $\mathcal{E}_{\lambda_2(\ones)}$ and an embedding $\mathcal{Q}$ on  $\mathcal{E}_{\lambda_n(\ones)}$ that are both edge-nonedge-isometric then 
    both $G$ and $G^c$ are conformally rigid. 
\end{corollary}

A class of graphs that have edge-nonedge-isometric embeddings on $\mathcal{E}_{\lambda_2(\ones)}$ and $\mathcal{E}_{\lambda_n(\ones)}$ are the connected strongly-regular graphs. 
If $G$ is a connected strongly-regular graph then there are only two distances among pairs of vertices in $G$; $d(i,j)=1$ if $ij \in E$ and $d(i,j)=2$ if $ij \not \in E$. Also $G$ has only two nonzero eigenvalues which are typically called $r$ and $s$. Complements of strongly regular graphs are also strongly-regular and suppose $r'=n-s$ and $s'=n-r$ are the nonzero eigenvalues of $G^c$. 
We will see in Corollary~\ref{cor:drgs have edge-iso-embeddings} that $G$ has a spherical edge-isometric embedding $\mathcal{U} = \{u_1, \ldots, u_n\}$ on $\mathcal{E}_{r}$ where $u_i^\top$ form the rows of a matrix $U$ whose columns form an orthonormal basis of $\mathcal{E}_r$. Since $\mathcal{E}_r = \mathcal{E}^c_{s'}$, 
 $\mathcal{U}$ is also a spherical embedding of $G^c$ on $\mathcal{E}^c_{s'}$ and 
by Corollary~\ref{cor:drgs have edge-iso-embeddings}, it is an edge-isometric embedding of $G^c$. Since the edges of $G^c$ are the nonedges of $G$ and vice-versa,  $\mathcal{U}$ is an edge-nonedge-isometric embedding of $G$ on $\mathcal{E}_r$ (and of $G^c$ on $\mathcal{E}^c_{s'}$). Similarly, there is an edge-nonedge-isometric spherical embedding $\mathcal{V}$ of 
$G$ on $\mathcal{E}_s = \mathcal{E}^c_{r'}$ (and $G^c$ on $\mathcal{E}^c_{r'}$). The embeddings $\mathcal{U}$ and $\mathcal{V}$ are spherical $2$-{\em designs} \cite[Corollary 6.2]{godsil-alg-comb-book}. Lemma 6.3 in \cite{godsil-alg-comb-book} says that if 
$\mathcal{U}$ is a spherical $2$-design such that $u_i^\top u_j \in \{\alpha, \beta\}$ whenever $u_i \neq u_j$, and $G$ is the graph with vertex set $\mathcal{U}$ and two vectors adjacent if their inner product is $\alpha$, then $G$ is strongly-regular. 

\begin{example} \label{ex:petersen}
    Let $G$ be the Petersen graph which is strongly-regular. It is easy to see that $G$ and $G^c$ are both conformally rigid since they are both edge-transitive. The eigenvalues of $G$ and $G^c$ with their multiplicities are: 
    $$G: (5, 4), (2, 5), (0, 1) \,\,\,\textup{ and } \,\,\,G^c: (8, 5), (5, 4), (0, 1).$$
    Take $\mathcal{U}$ to be the columns of 
    $${\small U^\top = \left(
\begin{array}{cccccccccc}
 \frac{1}{2} & -\frac{1}{2} & 0 & 0 & -\frac{1}{2} & \frac{1}{2} & 0 & 0 & 0 & 0 \\
 -\frac{1}{2 \sqrt{3}} & \frac{1}{2 \sqrt{3}} & -\frac{1}{\sqrt{3}} & 0 & -\frac{1}{2 \sqrt{3}} & \frac{1}{2 \sqrt{3}} & \frac{1}{\sqrt{3}} & 0 & 0 & 0 \\
 -\frac{1}{2 \sqrt{3}} & -\frac{1}{2 \sqrt{3}} & \frac{1}{2 \sqrt{3}} & -\frac{1}{\sqrt{3}} & 0 & 0 & \frac{1}{2 \sqrt{3}} & \frac{1}{\sqrt{3}} & 0 & 0 \\
 -\frac{1}{6} & -\frac{1}{6} & -\frac{1}{6} & \frac{1}{3} & -\frac{1}{3} & -\frac{1}{3} & -\frac{1}{6} & \frac{1}{3} & \frac{2}{3} & 0 \\
 -\frac{1}{3 \sqrt{2}} & -\frac{1}{3 \sqrt{2}} & -\frac{1}{3 \sqrt{2}} & -\frac{1}{3 \sqrt{2}} & \frac{1}{3 \sqrt{2}} & \frac{1}{3 \sqrt{2}} & -\frac{1}{3 \sqrt{2}} & -\frac{1}{3 \sqrt{2}} & \frac{1}{3 \sqrt{2}} & \frac{1}{\sqrt{2}} \\
\end{array}
\right)}.$$
Then $\mathcal{U}$ is an edge-nonedge-isometric embedding of $G$ on $\mathcal{E}_2$ and of $G^c$ on $\mathcal{E}^c_8 = \mathcal{E}_2$. Indeed, $\|u_i - u_j \| = \sqrt{2/3}$ for all $ij \in E$ and $\|u_i - u_j \| = 2/\sqrt{3}$ for all $ij \not \in E$.

Take $\mathcal{V}$ to be the columns of 
$${\small V^\top = \left(
\begin{array}{cccccccccc}
 \frac{1}{\sqrt{6}} & \frac{1}{\sqrt{6}} & 0 & -\frac{1}{\sqrt{6}} & -\frac{1}{\sqrt{6}} & -\frac{1}{\sqrt{6}} & 0 & 0 & 0 & \frac{1}{\sqrt{6}} \\
 \frac{1}{\sqrt{6}} & 0 & -\frac{1}{\sqrt{6}} & -\frac{1}{\sqrt{6}} & \frac{1}{\sqrt{6}} & 0 & 0 & 0 & \frac{1}{\sqrt{6}} & -\frac{1}{\sqrt{6}} \\
 \frac{1}{3 \sqrt{2}} & -\frac{1}{3 \sqrt{2}} & -\frac{\sqrt{2}}{3} & \frac{1}{3 \sqrt{2}} & \frac{1}{3 \sqrt{2}} & -\frac{1}{3 \sqrt{2}} & 0 & \frac{\sqrt{2}}{3} & -\frac{\sqrt{2}}{3} & \frac{1}{3 \sqrt{2}} \\
 \frac{1}{3 \sqrt{10}} & -\frac{2}{3} \sqrt{\frac{2}{5}} & \frac{1}{3 \sqrt{10}} & \frac{1}{3 \sqrt{10}} & \frac{1}{3 \sqrt{10}} & -\frac{2}{3} \sqrt{\frac{2}{5}} & \sqrt{\frac{2}{5}} & -\frac{2}{3} \sqrt{\frac{2}{5}} & \frac{1}{3 \sqrt{10}} & \frac{1}{3 \sqrt{10}} \\
\end{array}
\right)}.$$
Then $\mathcal{V}$ is an edge-nonedge-isometric embedding of $G$ on $\mathcal{E}_5$ and 
of $G^c$ on $\mathcal{E}^c_5 = \mathcal{E}_5$. In this case, 
$\|v_i - v_j \| = 2/\sqrt{3}$ for all $ij \in E$ and 
$\|v_i - v_j \| = \sqrt{2/3}$ for all $ij \not \in E$. 
The certificates are 
$$X = Y' = \frac{3}{2} UU^\top \,\,\textup{ and } \,\, Y= X' = \frac{3}{4} VV^\top.$$
The scalings $3/2$ and $3/4$ are determined from the edge lengths in the two embeddings as in Remark~\ref{rem:scaling X Y}.
The columns of $U$ (resp. $V$) form an  orthonormal basis of $\mathcal{E}_2=\mathcal{E}^c_8$ (resp. $\mathcal{E}_5=\mathcal{E}^c_5$). Since 
$\|u_i\| = 1/\sqrt{2}$  and $\|v_i\| = \sqrt{2/5}$ for all $i=1,\ldots, 10$, 
the embeddings $\mathcal{U}$ and $\mathcal{V}$ are both spherical. \qed
\end{example}

\section{Proofs}
\label{sec:proofs}

\subsection{The complete graph} \label{proofs:Kn}
We give a simple proof that the complete graph is conformally rigid.
This proof is the only proof where we show that the constant weights are the only choice of weights maximizing $\lambda_2$ and minimizing $\lambda_n$; it is structurally quite different from the other proofs in the paper and makes heavy use of the particular structure of the complete graph.
\begin{proof}[Proof of Proposition~\ref{prop:Kn}]
Consider a weighted complete graph on $n$ vertices with Laplacian $L_w$ such that 
$\sum_{ij \in E} w_{ij} = |E| = n(n-1)/2$.
If the smallest nonzero eigenvalue of $L_w$ is $n$, then all the nonzero eigenvalues of $L_w$ are $n$ since there is always one eigenvalue which is $0$ and 
$$ \mbox{tr}( L_w) = \sum_{i=1}^{n} \sum_{j \neq i} w_{ij} = n (n-1) = \sum_{i=1}^{n} \lambda_i(w) \geq (n-1) \lambda_2(w).$$
Likewise, if the largest nonzero eigenvalue is $n$, then all the nonzero eigenvalues are $n$ since there is always one eigenvalue which is $0$ and
$$ \mbox{tr}( L_w) = \sum_{i=1}^{n} \sum_{j \neq i} w_{ij} = n (n-1) = \sum_{i=1}^{n} \lambda_i(w) \leq (n-1) \lambda_n(w).$$
In either case, there are only two eigenspaces; $\textup{span}(\ones)$ with eigenvalue $0$ and $\ones^\perp$ with eigenvalue $n$. If $f:V \rightarrow \mathbb{R}$ has mean value 0 (i.e., $f \in \ones^\perp$), then
$$ L_wf = n f,$$
and if $f$ is constant, then $L_wf = 0$. 
The matrix $J \in \mathbb{R}^{n \times n}$ of all ones, 
plays the inverse role: if $f:V \rightarrow \mathbb{R}$ has mean value 0, then $J f = 0 \in \mathbb{R}^n$, and if $f$ is constant, then $J f = n f$. Therefore
$$ L_w + J = n \cdot I$$
which forces $w = \ones$.
\end{proof}

\subsection{Edge-transitive graphs} \label{proofs:edge-transitive}

\begin{proof}[Proof of Proposition~\ref{prop:edge-transitive}]
If $G$ is edge-transitive, then the SDP \eqref{eq:prob l2 sdp} is 
the one-variable symmetry reduced problem \eqref{eq:prob l2 sdp symm}:
\begin{align}
    \textup{min }\left\{ |E| \alpha \,:\,  \alpha L \succeq I - \frac{1}{n} J, \,\, \alpha  \geq 0\right\}.
\end{align}
The solution $\alpha = 1/\lambda_2(\ones)$ is optimal for this problem and hence, $w^* = \ones/\lambda_2(\ones)$ is optimal for \eqref{eq:prob l2 sdp}.
The argument for $\lambda_n(w)$ is analogous.

\end{proof}

\subsection{Distance-Regular Graphs}
This section provides some background material and then proves Theorem \ref{thm:distanceregular}. We start by recalling the basic setup.
We use $d(i,j)$ to denote the length of the shortest path between vertices $i$ and $j$ in a connected graph $G=([n],E)$. The diameter of $G$ is the largest distance between any pair of vertices, formally $\textup{diam}(G) = \max_{i,j \in [n]} d(i,j)$. 

\begin{definition} \label{def:DRG} 
A connected graph $G$ is 
called {\em distance-regular} if for any two vertices $i$ and $j$ 
$$ |\{u \in [n] \,:\,d(u,i)=p, \,\,d(u,j)=q   \}|$$
only depends on $d(i,j)$ and $p, q \in \mathbb{N}$. 
\end{definition}

Distance-regular graphs are always regular and the number of distinct distances $d(i,j)$ among the vertices equals $\textup{diam}(G)+1$ (since $d(i,i) = 0$). We also formally recall the definition of strongly-regular graphs.

\begin{definition}
    A regular graph $G=([n],E)$ that is not complete or empty is {\em strongly-regular} if any two adjacent vertices $i,j$ have $a$ common neighbors, and any two non-adjacent vertices have $b$ common neighbors.
    %If the valency of $G$ is $k$, we write $G$ as $G=SR(n,k,a,b)$.
\end{definition}

A connected strongly-regular graph is a distance-regular graph of diameter $2$. We will now quickly survey known results about distance-regular graphs and refer to \cite[Chapter 13]{godsil-alg-comb-book} for additional results and details.\\

Let $G=([n],E)$ be a distance-regular graph with Laplacian $L$, and $\lambda$ an eigenvalue of $L$ with multiplicity $m$. Consider an embedding $\mathcal{U}=\{u_1, \ldots, u_n\} \subset \RR^m$ of $G$ on $\mathcal{E}_\lambda$ as in Definition~\ref{def:embedding}, with the extra requirement that the columns of $U \in \RR^{n \times m}$ form an orthonormal basis of $\mathcal{E}_\lambda$. 
% We now consider an embedding of a distance-regular graph as in Definition~\ref{def:embedding} with the extra considition 
%  We denote the rows of $U$ by $u_1^\top, \ldots, u_n^\top$, these are $n$ vectors in $\mathbb{R}^m$. Following the language in \cite[Chapter 13]{godsil-alg-comb-book}, the set of these vectors, $\mathcal{U} = \{ u_1, \ldots, u_n \} \subset \RR^m$, is called a {\em representation} of $G$ on $\mathcal{E}_\lambda$. We also may think of it as a way how to embed the $n$ vertices of the graph in $m-$dimensional space. 
The projection matrix $UU^\top$ (from $\RR^n \rightarrow \mathcal{E}_\lambda$) is 
the {\em Gram matrix} of $\mathcal{U}$ with $ij$-entry equal to 
$u_i^\top u_j$. %These notions are meaningful for arbitrary graphs. 
When $G$ is distance-regular, the embedding $\mathcal{U}$ is very rigid as formalized in the following Lemma.

\begin{lemma} \cite[Lemma 1.2]{godsil-alg-comb-book} \label{lem:drg innerproducts}
Let $G$ be a distance-regular graph and $\mathcal{U}$ be an embedding of $G$ on the eigenspace $\mathcal{E}_\lambda$ as described above. If $i$ and $j$ are two vertices of $G$, then the inner product $u_i^\top u_j$ is determined by $d(i,j)$. 
\end{lemma}

We can use this Lemma to show that, for distance-regular graphs, $\mathcal{U}$ is an edge-isometric embedding.

\begin{corollary} \label{cor:drgs have edge-iso-embeddings}
    If $G$ is a connected distance-regular graph and $\lambda > 0$, then $\mathcal{U}$ is an 
    edge-isometric spherical embedding of $G$ on $\mathcal{E}_\lambda$. 
\end{corollary}

\begin{proof}
    By construction, $\mathcal{U}$ is an embedding of $G$ on 
    $\mathcal{E}_\lambda$. An immediate consequence of  Lemma~\ref{lem:drg innerproducts} is that, since $d(i,i) = 0$,  $\|u_i\|^2 = u_i^\top u_i$ is independent of $i$, and hence all $u_i$ have the same $\ell^2-$norm. Since the columns of $U$ are orthonormal, the square of its Frobenius norm is 
$$ \|U\|_F^2 = \sum_{i,j} u_{ij}^2 = m \qquad \mbox{and thus} \qquad\|u_i\| = \sqrt{\frac{m}{n}}.$$
In particular, all the points in $\left\{u_1, \dots, u_n\right\} \subset \mathbb{R}^m$ are on a sphere centered at the origin.
 It remains to show that if $(i,j) \in E$, that $\|u_i - u_j\|$ is a positive constant independent of $i,j$. If $(i,j) \in E$, then $d(i,j)=1$. Applying Lemma~\ref{lem:drg innerproducts} one more time, we see that there exists a constant $c_2$ such that
$$ (i,j) \in E \implies  u_i^\top u_j = c_2.$$
Thus, whenever $(i,j) \in E$, we have
\begin{align*}
    \|u_i - u_j\|^2 &= \|u_i\|^2 + \|u_j\|^2 - 2 u_i^\top u_j = \frac{2m}{n} - 2 c_2
\end{align*}
    which is indeed a nonnegative constant independent of $i$ and $j$. It remains to show that $2m/n - 2c_2$ is a \textit{positive} constant. We argue by contradiction. If $2m/n - 2c_2$ was $0$, then we would have $u_i = u_j$ for all $(i,j) \in E$. Since the graph is connected, this means that $u_i = u_j$ for all $1 \leq i,j \leq n$. This contradicts the fact that the columns of $U$ are an orthonormal set as soon as there are at least two columns. So, suppose $U \in \mathbb{R}^{n \times 1}$ which happens if the dimension of the eigenspace $\mathcal{E}_{\lambda}$ is 1. Appealing again to Lemma~\ref{lem:drg innerproducts}, we see that all entries of $U$ are $\pm \sqrt{1/n}$. Since the
the column $U$ is an eigenvector of the graph Laplacian and $\lambda > 0$, orthogonality to the constant vector implies that both numbers $\pm \sqrt{1/n}$ have to appear (and, in fact, have to appear the same number of times), which forces $n$ to be even. This suggests the decomposition $V = [n] = A \cup B$ where
$$ A = \left\{i \in [n]: u_i = - \frac{1}{\sqrt{n}}\right\} \qquad \mbox{and} \qquad B = \left\{i \in [n]: u_i =  \frac{1}{\sqrt{n}}\right\}.$$
Moreover, appealing once more to Lemma~\ref{lem:drg innerproducts}, if $(i,j) \in E$, then $u_i u_j = c_3$ for some constant $c_3$ that is independent of $i$ and $j$. This constant has to be either $1/n$ or $-1/n$. If it were $1/n$, then this would mean that there is not a single edge between $A$ and $B$ which would mean that the graph is not connected which is a contradiction. Therefore, $c_3 = -1/n$. This, in turn, means that all edges run between $A$ and $B$ which shows that the embedding is indeed edge-isometric.
\end{proof}

The embedding $\mathcal{U}$ is in fact a 
spherical $2$-design \cite[Corollary 6.2]{godsil-alg-comb-book}.

\begin{proof}[Proof of Theorem \ref{thm:distanceregular}]
   Proposition~\ref{prop:embedding certificates} shows that conformal rigidity is equivalent to having suitable edge-isometric embeddings of the graph on $\mathcal{E}_{\lambda_2}$ and $\mathcal{E}_{\lambda_n}$ which is guaranteed by the preceding Corollary. 
\end{proof}

We note that this entire approach is specific to distance-regular graphs: one does not in general obtain edge-isometric spherical embeddings via this approach (though, some of the ideas can be used to produce certificates at a greater level of generality, see \S \ref{sec:uut}). 
For instance, the claw graph on $4$ vertices is edge-transitive and hence conformally rigid. Its Laplacian has two nonzero eigenvalues and the 
embedding $\mathcal{U}$ constructed from an orthonormal basis of either eigenspace is edge-isometric, but not spherical.

\subsection{Cayley graphs}
\begin{proof}[Proof of Theorem \ref{thm:Cayley}]
Suppose there exists an assignment of weights such that $\lambda_2(w) > \lambda_2(\ones)$. This means that for every function $f:G \rightarrow \mathbb{R}$ with mean value 0 and normalized in $\ell^2(\Gamma)$, we have, for non-constant weights $w_{g,s}$ on the edge  connecting $g$ and $g \circ s$,
$$ \sum_{s \in S} \sum_{g \in G} w_{g,s} (f(g) - f(g \circ s))^2 \geq \lambda_2(w) > \lambda_2(\ones).$$
We use $\phi$ to denote an eigenvector corresponding to $\lambda_2(\ones)$ when $w_{g,s} = 1$ (and we also assume it to be normalized in $\ell^2$).
We introduce
$$ \phi_a(v) = \phi(a \circ v).$$
Then, from the above inequality we have 
$$ \sum_{a \in G} \sum_{s \in S} \sum_{g \in G} w_{g,s} (\phi_a(g) - \phi_a(g \circ s))^2 \geq (\# G) \cdot \lambda_2(w).$$
We can rewrite the left hand side as
\begin{align*}
 \sum_{a \in G} \sum_{s \in S} \sum_{g \in G} w_{g,s} (\phi_a(g) - \phi_a(g \circ s))^2 &=   \sum_{s \in S} \sum_{g \in G}  \sum_{a \in G} w_{g,s} (\phi_a(g) - \phi_a(g \circ s))^2\\
 &=    \sum_{s \in S} \sum_{g \in G}  \sum_{a \in G} w_{a^{-1} g,s} (\phi(g) - \phi(g \circ s))^2 \\
 &=  \sum_{s \in S} \sum_{g \in G}  \left( \sum_{a \in G} w_{a^{-1} g,s}\right) (\phi(g) - \phi(g \circ s))^2.
\end{align*}
This means that we can replace the weight $w_{g, s}$ by the average in the entire orbit
$$ w_{g, s} \rightarrow \frac{1}{\# G} \sum_{a \in G} w_{a g,s}$$
and still obtain a function with the desired property of exceeding $\lambda_2(w) > \lambda_2(\ones)$. Thus we may assume that $w_{g,s}$ only depends on $s$, and 
$$ \sum_{s \in S} w_{s} \sum_{g \in G}  (\phi(g) - \phi(g \circ s))^2 \geq \lambda_2(w) > \lambda_2(\ones).$$
Squaring out, we see that (invoking the $\ell^2-$normalization of $\phi$)
\begin{align*}
    \sum_{g \in G}  (\phi(g) - \phi(g \circ s))^2 &= \sum_{g \in G}  \phi(g)^2 + \phi(g \circ s)^2 - 2 \phi(g)  \phi(g \circ s) \\
    &= 2 - 2 \sum_{g \in G}\phi(g)  \phi(g \circ s).
\end{align*}
Now suppose this last sum is independent of $s$, i.e., for some $c>0$ and all $s \in S$
$$ c =   \sum_{g \in G}  (\phi(g) - \phi(g \circ s))^2.$$
Therefore
$$ \sum_{s \in S} w_{s} \sum_{g \in G}  (\phi(g) - \phi(g \circ s))^2 =  c\sum_{s \in S} w_{s}.$$
A sum of finitely many real numbers does not change if we replace each summand by the global average and thus
\begin{align*}
    \sum_{s \in S} w_{s} \sum_{g \in G}  (\phi(g) - \phi(g \circ s))^2 &=  c\sum_{s \in S} w_{s} =  c\sum_{s \in S} \left( \frac{1}{|S|}\sum_{t \in S} w_{t}\right)    \\
   &= \sum_{s \in S} \left( \frac{1}{|S|}\sum_{t \in S} w_{t}\right)\sum_{g \in G}  (\phi(g) - \phi(g \circ s))^2\\
   &= \sum_{s \in S} \sum_{g \in G}  \left( \frac{1}{|S|}\sum_{t \in S} w_{t}\right) (\phi(g) - \phi(g \circ s))^2.
\end{align*}
However, at this point the weights are independent of $s$ and thus constant. This is a contradiction to $\lambda_2(w) > \lambda_2(\ones)$. The argument for $\lambda_n$ is the same.
\end{proof}

\subsection{Proof of Proposition~\ref{prop:circ}}
The proof makes use of the fact that eigenvalues and eigenvectors of circulant graphs are well understood. In particular, in the case when all the weights are equal, 
$$ \lambda_2( C_n\left(\left\{1,2\right\}\right)) = 4 - 2 \cos\left( \frac{2\pi}{n} \right) - 2 \cos\left(\frac{4 \pi}{n} \right).$$
One way of seeing this is to realize that the graph Laplacian can be written as a circulant matrix
$$ L = \begin{pmatrix}
    c_0 &c_{n-1} & c_{n-2} & \dots & c_1 \\
    c_1 & c_0 & c_{n-1} & \dots & c_2 \\
   c_2 & c_1 &c_0 & \dots & c_3 \\
    \vdots & \vdots & \vdots & \ddots & \vdots \\
    c_{n-1} & c_{n-2} & c_{n-3} & \dots & c_0 \\
  \end{pmatrix}$$
where $c_0 = 4$ and $c_1 = c_2 = c_{n-1} = c_n = -1$ while all other $c_i$ are 0. It is known that the eigenvalues of such a circulant matrix are given by
$$ \lambda_j = c_0 + c_1 \omega^j + c_2 \omega^{2j} + \dots + c_{n-1} \omega^{(n-1)j} \qquad \mbox{where}~\omega = \exp\left(\frac{2\pi i}{n} \right),$$
where $j$ runs from $0 \leq j \leq n-1$.
In our case, this expression simplifies to
\begin{align*}
    \lambda_j &= 4 - \omega^j - \omega^{2j} - \omega^{(n-1)j} - \omega^{(n-2)j} \\
    &= 4 - \cos\left(\frac{2\pi j}{n} \right) - \cos\left(\frac{4\pi j}{n} \right)  - \cos\left(\frac{2\pi(n-1) j}{n} \right)  -\cos\left(\frac{2\pi (n-2) j}{n} \right) \\
    &= 4 - 2\cos\left(\frac{2\pi j}{n} \right) - 2\cos\left(\frac{4\pi j}{n} \right).
\end{align*}
The case $j=0$ corresponds to the eigenvalue $\lambda_1=0$. To make the cosine as large as possible, we need the arguments in the cosine to be as close as possible to a multiple of $2\pi$ and this leads to $j=1$ and $j=n-1$. We will now investigate a different choice of weights on the circulant that maintains the circulant structure. For a $\varepsilon>0$ to be thought of as very small, we set $c_0 = 4$ as well as
$$ c_1 = c_{n-1} = -1 + \varepsilon \qquad \mbox{and} \qquad c_2 = c_{n-2} = -1-\varepsilon.$$
Then
\begin{align}\label{eq:lambdaj}
\lambda_j = 4 - (2-\varepsilon)\cos\left(\frac{2\pi j}{n} \right) - (2+\varepsilon)\cos\left(\frac{4\pi j}{n} \right).
\end{align}
Our goal will be to show that for a small value of $\varepsilon$ we can ensure that $\lambda_2$ increases. However, it is clear that some restriction on $n$ will be required.
When $n=5$, then  $C_5\left(\left\{1,2\right\}\right) = K_5$ is conformally rigid. Moreover, $C_6\left(\left\{1,2\right\}\right)$ is the octahedral graph which is conformally rigid (it is edge-transitive). In both of these cases, the procedure cannot work and it is easy to see from the spectrum why: $K_5$ has eigenvalues $(0, 5, 5, 5, 5)$ and a perturbation induced by $\varepsilon$ affects all eigenvalues and some will increase while others decrease. Something similar is happening when $n=6$, the octahedral graph has spectrum $(0,4,4,4,6,6)$, in particular, the multiplicity of the second eigenvalue is 3, this allows for it to be conformally rigid.
\begin{lemma}
    For every $n \geq 7$ there exists $\varepsilon_0 > 0$ so that for all $\varepsilon \in (0, \varepsilon_0)$, the multiplicity of $\lambda_2$ is 2.
\end{lemma}
\begin{proof}
    It suffices to prove the result when $\varepsilon=0$ and then appeal to continuity of the expression in $\varepsilon$. Note first that the expression \eqref{eq:lambdaj} is the same for $j$ and $-j$ and the expression is $n-$periodic and thus it is the same for $j$ and $n-j$, so it suffices to understand the behavior when $j \leq n/2$. We introduce the function
    $$ f(x) = 4 - 2 \cos\left(2 \pi x\right) - 2 \cos\left(4 \pi x\right)$$
    and it suffices to understand this function for $0 \leq x \leq 1/2$ (see Fig. \ref{fig:cosine}).
    A quick computation shows that
    $$ f\left(\frac{1}{6}\right) = f\left(\frac{1}{2}\right) =4.$$
    This appearance of $1/6$ corresponds exactly to the conformal rididity of the octahedral graph.
    The function is monotonically increasing on $(0, 1/6)$ and thus, once $n \geq 7$, we have
    $$ \min_{1 \leq j \leq n/2} f\left(\frac{j}{n}\right) = f(1/n)$$
    which is attained only when $j=1$ and the remaining case arises, from symmetry, when $j = n-1$.
    This proves the result for $\varepsilon =0$, the existence of an entire range $(0,\varepsilon_0)$ follows from continuity.
\end{proof}
    \begin{center}
        \begin{figure}[h!]
            \centering
            \begin{tikzpicture}
             \node at (0,0) {   \includegraphics[width=0.2\textwidth]{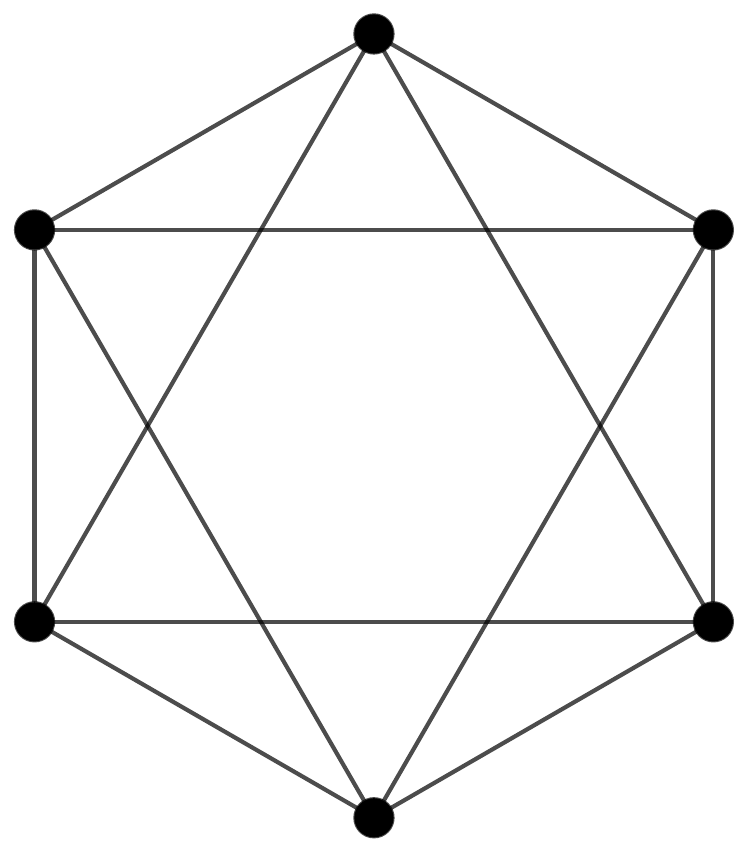}};  
              \node at (5,0) {   \includegraphics[width=0.4\textwidth]{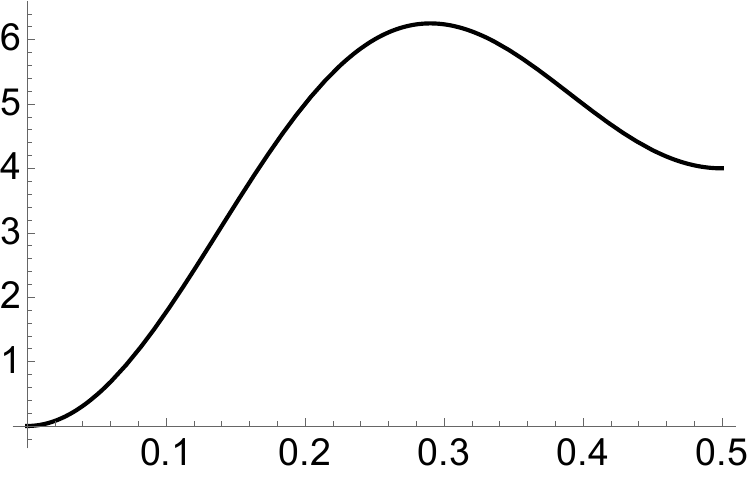}};  
              \draw [dashed]   (4.2,0.46) -- (7.3,0.46);
            \end{tikzpicture}
            \caption{The octahedral graph $C_6\left(\left\{1,2\right\}\right)$ and the function $f(x)$ for $0 \leq x \leq 1/2$ with level set $f(1/6) = f(1/2) =4$.}
            \label{fig:cosine}
        \end{figure}
    \end{center}
Having proved this Lemma, the argument now follows relatively easily. We have once $n \geq 7$, for an entire range $0 \leq \varepsilon \leq \varepsilon_0$ that
$$ \lambda_2(\varepsilon)= 4 - (2-\varepsilon)\cos\left(\frac{2\pi }{n} \right) - (2+\varepsilon)\cos\left(\frac{4\pi }{n} \right).$$
This expression is increasing and thus there is a way of changing the weights to increase $\lambda_2$ which shows that $C_n\left(\left\{1,2\right\}\right)$ is not conformally rigid for $n \geq 7$.

\section{Certifying Conformal Rigidity via SDPs} \label{sec:sdp certificates}
In this section we certify the conformal rigidity of several graphs mentioned in \S~\ref{subsec:isolated} via semidefinite programming (see \S~\ref{sec:SDP}). The solutions returned by an SDP solver are typically numerical and only approximately true. Often with some massaging they can be turned into verifiable proofs of conformal rigidity. We explain our methods, which rely on the theory explained in \S~\ref{sec:SDP}.

\subsection{The Rationalization Trick.} 
Solving the SDP \eqref{eq:prob l2 sdp} numerically on a computer, we obtain an approximate (primal) solution, an approximate (dual) solution, and a duality gap. In practice, we considered a duality gap of $< 10^{-7}$ as reasonable numerical evidence that the graph under consideration is indeed conformally rigid. However, this is not a proof. 
If the graph is conformally rigid, then the constant assignment of weights
$$ w_{ij} = \frac{1}{\lambda_2(G)} \qquad \mbox{and} \qquad w_{ij} = \frac{1}{\lambda_n(G)} $$
is an optimal (primal) solution for the SDPs \eqref{eq:prob l2 sdp} and \eqref{eq:prob ln sdp} respectively. Their optimality can be certified 
by optimal solutions $X$ and $Y$ to the dual SDPs \eqref{eq:prob l2 sdp dual} and 
\eqref{eq:prob ln sdp dual}, via strong duality. We first obtain  approximate numerical dual optimal solutions $X$ and $Y$ using the SDP solver MOSEK \cite{mosek} which we called through Mathematica \cite{mathematica}. If there is an exact solution comprised of rational numbers, then approximating the numerical entries by nearby rational numbers can lead to a closed form expression for true optimal solutions. We found this method to be remarkably effective in cases where $\lambda_2(\ones)$ and $\lambda_n(\ones)$ were both rational numbers.\\

We first illustrate the procedure for the Hoffman graph (see Fig.~\ref{fig:cert hoffman}). The solver produced a dual solution $X$,  a matrix with entries like -0.434896 and 0.565104. A little bit of an inspection shows that only five different numbers truly appear and that the numbers are exceedingly close to
$$\mbox{the rational numbers} \qquad \left\{ -\frac{359}{384}, -\frac{167}{384}, \frac{25}{384}, \frac{217}{384}, \frac{409}{384} \right\}.$$
Replacing the numerical entries in the matrix by these entries (see Fig.~\ref{fig:cert hoffman}), we can quickly check (symbolically) that this is indeed a dual optimal $X$. The same procedure can be repeated for $\lambda_n$ where the numerical dual solution only consists of the numbers $-0.25$ and $0.25$ arranged in a block structure (see Fig.~\ref{fig:cert hoffman}). Again, once a candidate has been identified, verifying that it is an optimal $Y$ is easy.

\begin{center}
    \begin{figure}[h!]
    \begin{tikzpicture}
            \node at (-4,0) {\includegraphics[width=0.3\textwidth]{hoffman.pdf}};
        \node at (0,0) {\includegraphics[width=0.3\textwidth]{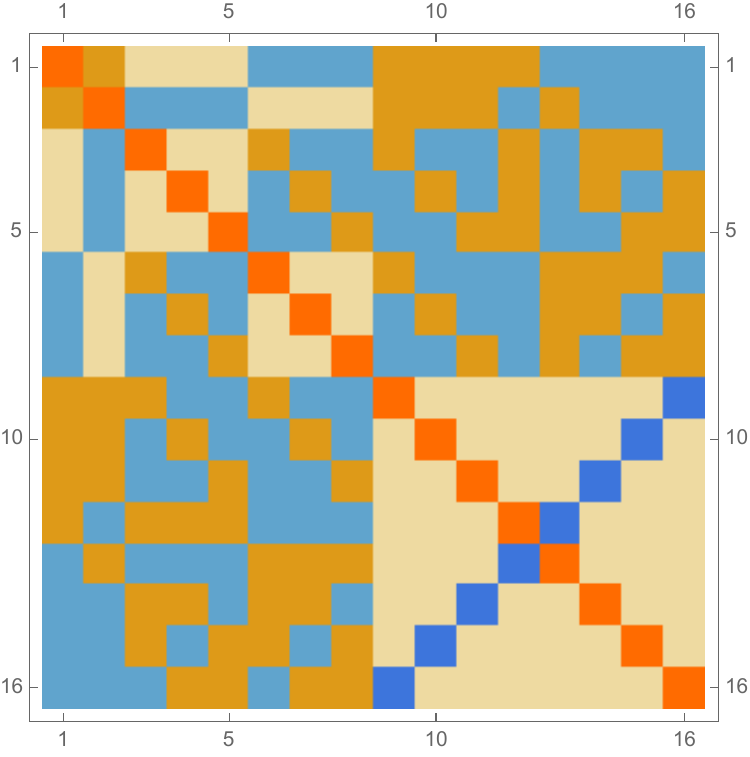}};
        \node at (4,0) {\includegraphics[width=0.3 \textwidth]{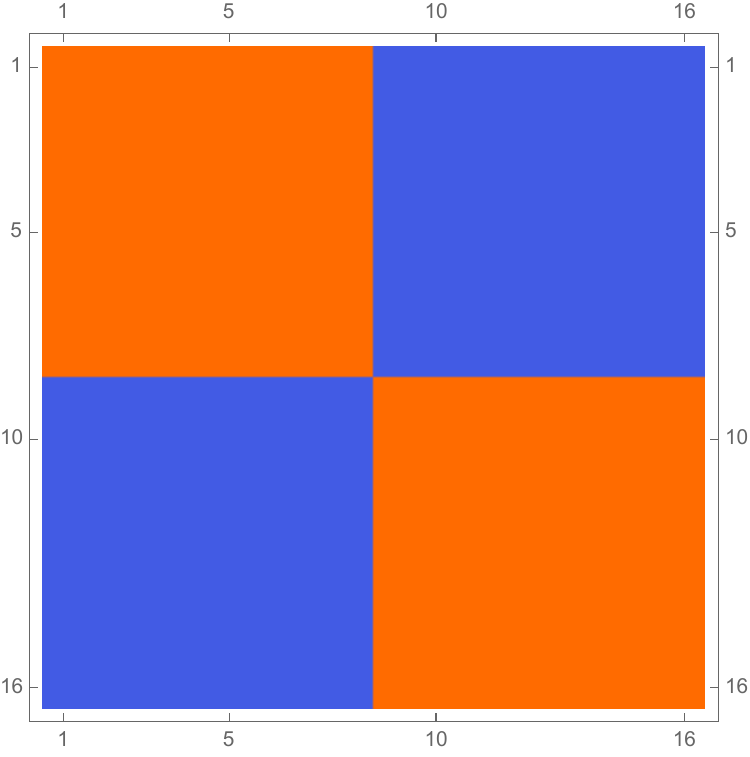}};
    \end{tikzpicture}
    \caption{Left: the Hoffman graph. Middle: the dual optimal solution $X$ for $\lambda_2$. Right: the dual optimal solution $Y$ for $\lambda_n$.}
    \label{fig:cert hoffman}
    \end{figure}
\end{center}

It should be noted that the two aspects that may cause complications are as follows: it is not a priori clear whether there is a dual optimum that is only comprised of rational numbers and, even if there is, it is not a priori clear how large the denominators are (equivalent to asking what type of accuracy one requires of the numerical approximation). On the other hand, once a candidate for an optimal solution has been identified, it is easy to prove that it is indeed one. 

% \begin{center}
%     \begin{figure}[h!]
%     \begin{tikzpicture}
%             \node at (-4,0) {\includegraphics[width=0.3\textwidth]{shri.pdf}};
%         \node at (0,0) {\includegraphics[width=0.3\textwidth]{shri1.pdf}};
%         \node at (4,0) {\includegraphics[width=0.3 \textwidth]{shri2.pdf}};
%     \end{tikzpicture}
%     \caption{Left: the complement of the Shrikhande graph. Middle: the certificate for $\lambda_2$. Right: the certificate for $\lambda_n$.}
%     \label{fig:cert2}
%     \end{figure}
% \end{center}

 % Our second example (see Figure \ref{fig:cert2}) is the complement of the Shrikhande graph which is conformally rigid but not edge-transitive. A quick look at the numerical solution for the dual SDP \eqref{eq:prob l2 sdp dual} only contains 3 values and it is easy to suspect that they are $-11/64, 5/64$ and $37/64$. Verifying that this leads to a correct certificate is then easy. The dual certificate for $\lambda_n$ has a comparably complex structure but contains even simpler entries, those being $-1/8, 1/8$ and $3/8$.

\begin{center}
    \begin{figure}[h!]
    \begin{tikzpicture}
            \node at (-4.1,0) {\includegraphics[width=0.3\textwidth]{noncrossing.pdf}};
        \node at (0,0) {\includegraphics[width=0.3\textwidth]{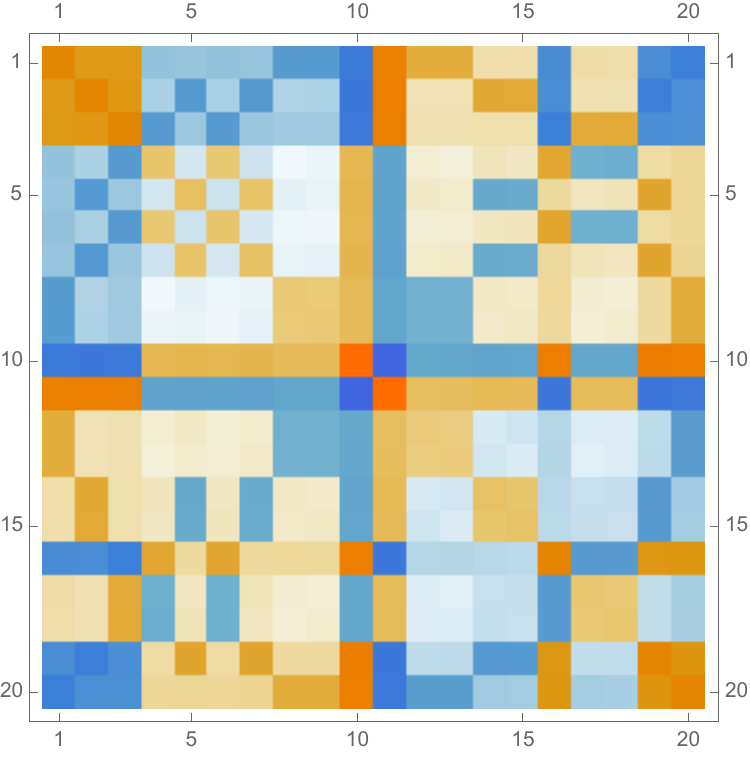}};
        \node at (4.1,0) {\includegraphics[width=0.3 \textwidth]{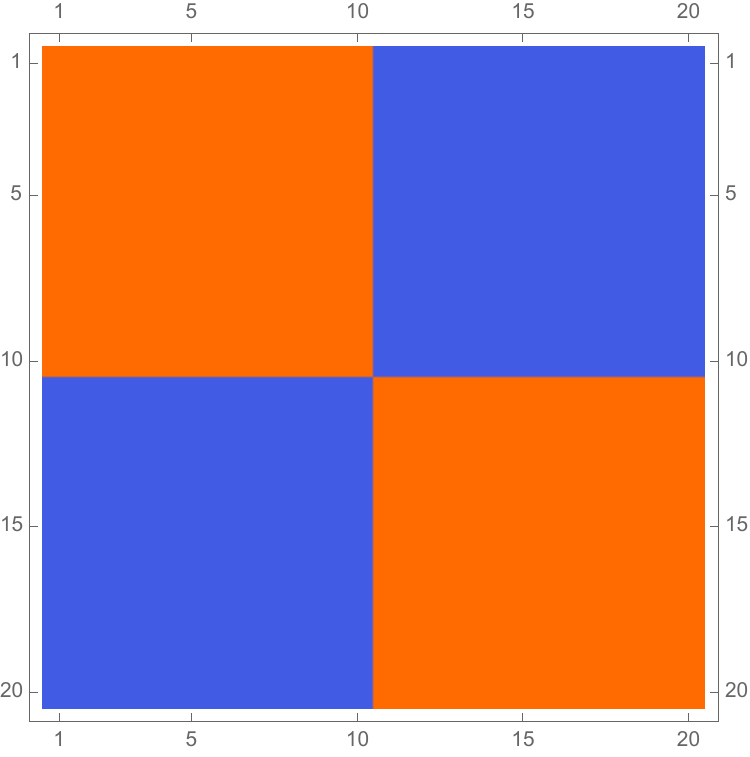}};
    \end{tikzpicture}
    \caption{Left: the CrossingNumberGraph6B. Middle: a numerical certificate for $\lambda_2$. Right: the certificate for $\lambda_n$.}
    \label{fig:cert3}
    \end{figure}
\end{center}

Another example is, CrossingNumberGraph6B, where the rationalization method by itself is not successful (see Fig.~\ref{fig:cert3}). The Laplacian 
has eigenvalues $\lambda_2 =1$ (with multiplicity 3) and $\lambda_n = 6$ (with multiplicity 1). While the rationalization method works for $\lambda_n$ (where the solution is particularly easy), it does not work for $\lambda_2$. In the next section we will use a more sophisticated method   to obtain an optimal solution $X$. It was mentioned in \cite{sun-boyd-xiao-diaconis} that if $X$ is feasible for the dual SDP \eqref{eq:prob l2 sdp dual}, then so is 
$(I-(1/n) J) X (I - (1/n)J)$. Same applies to $Y$. 
For many graphs, applying this transformation to the dual solutions  returned by the SDP solver was enough to make them nice and verifiable.

 \subsection{The Projection Method.} There is an alternate method that makes use of equations \eqref{eq:cs 2} and \eqref{eq:cs 4};  the rows (and columns) of the dual optima $X, Y$ are linear combinations of eigenvectors of the Laplacian matrix corresponding to $\lambda_2$ and $\lambda_n$, respectively. This suggests a natural idea.
 \begin{enumerate}
     \item Compute a numerical approximation for the dual optimum. 
     \item Project each row into the appropriate eigenspace of the Laplacian matrix.
     \item Convert the coefficients into reasonable nearby numbers (either rational numbers or algebraic numbers of low degree, like $\sqrt{2}$).
     \item Use this to recover an actual solution. 
 \end{enumerate}
We illustrate on the solution $X$ (for $\lambda_2$) for CrossingNumberGraph6B. The first row of a numerical candidate starts as
$$ (2.50568, 1.67231, 1.67231, 0.00569952, \dots )$$
and one might be tempted to think of these numbers as $5/2, 5/3, 5/3, 0, \dots$ but this does not work. However, there exists a particularly nice basis for the smallest eigenspace of $\lambda_2$ given by the vectors
\begin{align*}
    u_1 &=(0, 0, 4, -3, 1, -3, 1, 1, 1, -2, 2, -1, -1, -1, -1, -4, 3, 3, 0, 0) \\
  u_2 &= (0, -3, -1, 0, 2, 0, 2, -1, -1, 2, -2, 1, 1, -2, -2, 1, 0, 0, 3, 0)\\
  u_3 &=(-3, 0, -1, 0, -1, 0, -1, 2, 2, 2, -2, -2, -2, 1, 1, 1, 0, 0, 0, 3).
\end{align*}
Note that these vectors are not orthogonal but they are in closed form. Computing the inner product of the first row of the numerical $X$ with these three vectors, we get the numbers $(19.9994, -19.9994, -34.9999)$ and it's pretty clear what these are supposed to be. Correcting them and solving the linear system, we obtain a corrected matrix whose first row starts with
$$ \frac{7}{3}, \frac{3}{2}, \frac{3}{2}, -\frac{1}{6}, -\frac{1}{6}, \dots$$
which isn't even particularly close to the original numerical values. This then turns out to indeed be a correct dual optimal solution $X$.

This method also worked for the $(20,8)-$accordion graph (see Fig. \ref{fig:out}); the solution $X$ for $\lambda_n$ is trivial, the solution $Y$ for $\lambda_2$ is not. The numerical solution returned by the solver is not very insightful, however, projecting into the eigenspace corresponding to $\lambda_2$ is very helpful because it only has multiplicity 2: the new coordinates happened to either be 0 or, in case they were both nonzero, happen to have a ratio of $(1+\sqrt{5})/2$ which helped identify the correct solution.

If the multiplicity of either $\lambda_2$ or $\lambda_n$ is one, then there is a particularly easy dual optimum constructed from an eigenvector of that eigenvalue which is a special case of the projection method.
For the Haar Graph 565 shown in Fig.~\ref{fig:Haar} with $20$ vertices, we used the projection method to find $X$; this graph has 
$\lambda_2 = 2.76393$ with multiplicity $4$. 
However, $\lambda_{20} = 10$ with multiplicity $1$, and so we can use (\ref{obs:rank one certificate}) to find an optimal $Y \in \mathcal{S}^{20}_+$ for \eqref{eq:prob ln sdp} which has optimal value $q^*=5$ (see the next subsection). 
An eigenvector of $\lambda_{20}$ is 
$$u=(-1, -1, -1, -1, -1, -1, -1, -1, -1, -1, 1, 1, 1, 1, 1, 1, 1, 1, 1, 1)$$
% $$u = (-1, 1, 0, 1, 0, 1, 0, 1, 0, 0, 1, -1, -1, -1, 0, 0, -1, 0, 0, 0, -1, 
% 0, 0, 0, 0, 0, 1)$$
which has norm $\sqrt{20}$.
Therefore $Y = 5(1/20) uu^\top$ is a dual optimal solution for \eqref{eq:prob ln sdp} with trace equal to 
$5= q^*$.

\subsection{Symmetry Reduction} \label{subsec:symmetry}
Conformal rigidity of the Haar Graph 565 is especially easy via symmetry reduction. Haar Graph 565 has $50$ edges which fall into $2$ orbits of size $40$ and $10$ under the automorphism group of the graph. This Haar graph and its edge orbit graphs are shown in Fig.~\ref{fig:Haar}.

\begin{figure}[h!]
\begin{tikzpicture}
\node at (0,0) {\includegraphics[width=0.29\textwidth]{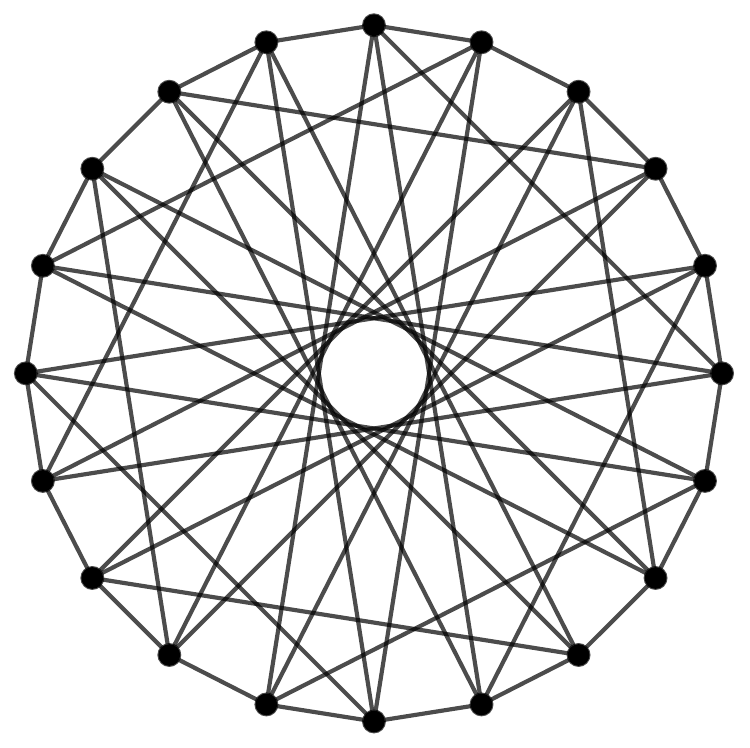}};
 \node at (4.5,0) {\includegraphics[width=0.29\textwidth]{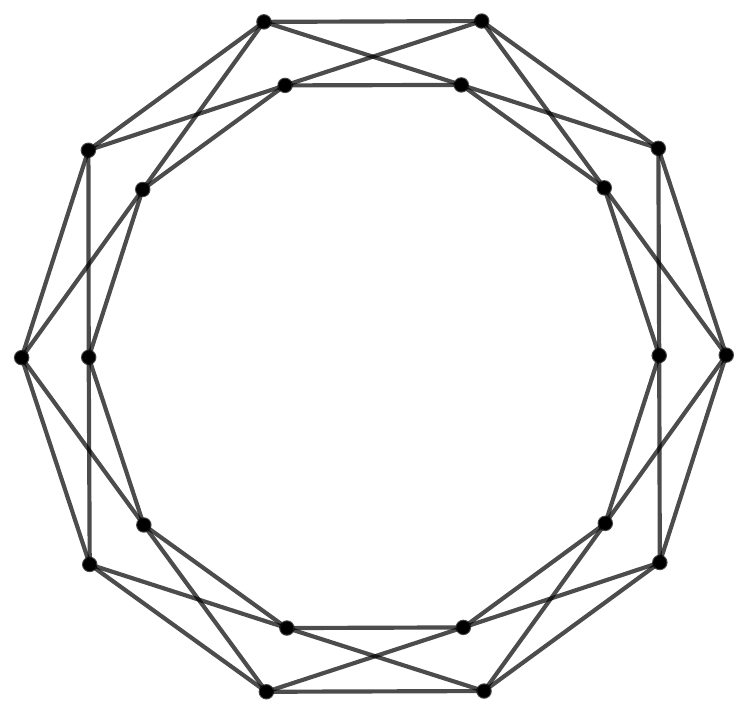}};
\node at (9,0) {\includegraphics[width=0.29\textwidth]{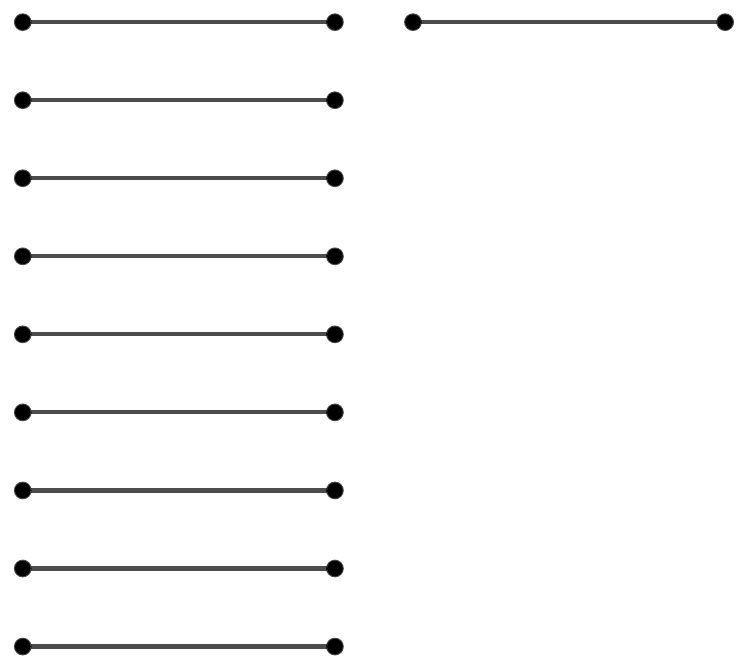}};
\end{tikzpicture}
\caption{Haar 565 (left) and its decomposition (middle and right) under the automorphism group of the graph.}
\label{fig:Haar}
\end{figure}

The symmetry reduced SDP \eqref{eq:prob l2 sdp symm} has $2$ variables and the constraint 
$$w_1 L_1 + w_2 L_2 \succeq I - \frac{1}{20} J.$$
The eigenvalues of $w_1 L_1 + w_2 L_2 - I + \frac{1}{20} J$ are linear in $w_1,w_2$ as shown below and can be computed explicitly in Mathematica: 
\begin{align} \label{eq:eigenvals Haar}
\begin{split}
0,\, -1 + 4 w_1, \,-1 + (5 \pm \sqrt{5}) w_1, \, 
%-1 + (5 +  \sqrt{5}) w_1,\\
-1 + (3 \pm \sqrt{5}) w_1 + 2 w_2, \\ 
%-1 + (3 + \sqrt{5}) w_1 + 2 w_2
 -1 + 4 w_1 +  2 w_2, -1 + 8 w_1+ 2 w_2
 \end{split}
 \end{align}
 Therefore, problem \eqref{eq:prob l2 sdp} is to minimize $40w_1 + 10w_2$ subject to all eigenvalues in \eqref{eq:eigenvals Haar} being nonnegative and $w_1,w_2 \geq 0$. This is a $2$-variable linear program with an optimal solution 
 $$ w_1 = w_2 = \frac{1}{\lambda_2(\ones)} \sim 0.361803.$$
 Therefore, $\ones/\lambda_2(\ones)$ is an optimal solution of \eqref{eq:prob l2 sdp}. To solve \eqref{eq:prob ln sdp}, we maximize $40w_1 + 10w_2$ subject to all the 
 eigenvalues in \eqref{eq:eigenvals Haar} being nonpositive and $w_1, w_2 \geq 0$. This is again a $2$-variable linear program with an optimal solution 
 $w_1 = w_2 = \frac{1}{10}$ and optimal value $q^* = 5$. Therefore, $\ones/\lambda_{20}(\ones)$ is an optimal solution of \eqref{eq:prob ln sdp} and Haar Graph 565 is conformally rigid. The solver we used returned the optimal solution $w_1 = 0, w_2 = 1/2$ which is not useful to 
 prove conformal rigidity but is useful to verify the optimality of $w_1 = w_2 = 1/10$. It also provides an example of a graph where a multiple of $\ones$ is not the only weight that minimizes $\lambda_n(w)$. 

A more involved example is the distance-$2$ graph of the Klein graph (see Fig. \ref{fig:out}) which has $n=24$ vertices and $168$ edges that come in $2$ orbits each of size $84$. In this case, 
$\lambda_2(\ones)=11.3542$ and 
$\lambda_n(\ones) = 16.6458$, each with multiplicity $8$.
The eigenvalues of $w_1 L_1 + w_2 L_2 - I + \frac{1}{24}J$ are:
\begin{align}
    \begin{split}
        0, \,\, -1 + 8w_1 + 8w_2, \,\, 
        -1 + 7w_1 + 7w_2 \pm  \sqrt{7(w_1^2 - w_1 w_2  + w_2^2)}\\
    \end{split}
\end{align}
Therefore, Problem~\eqref{eq:prob l2 sdp} is to 
minimize $84w_1 + 84w_2$ subject to the nonnegativity of the eigenvalues and $w_1,w_2 \geq 0$. This is a $2$-dimensional convex optimization problem with optimal solution $w_1 = w_2 =1/\lambda_2(\ones) \sim 0.0880728$. The feasible region and the optimal position of the objective function are shown in Fig.~\ref{fig:KleinDistance2}.
\begin{figure}[h!]
\begin{tikzpicture}
    \node at (0,0) {\includegraphics[scale=0.4]{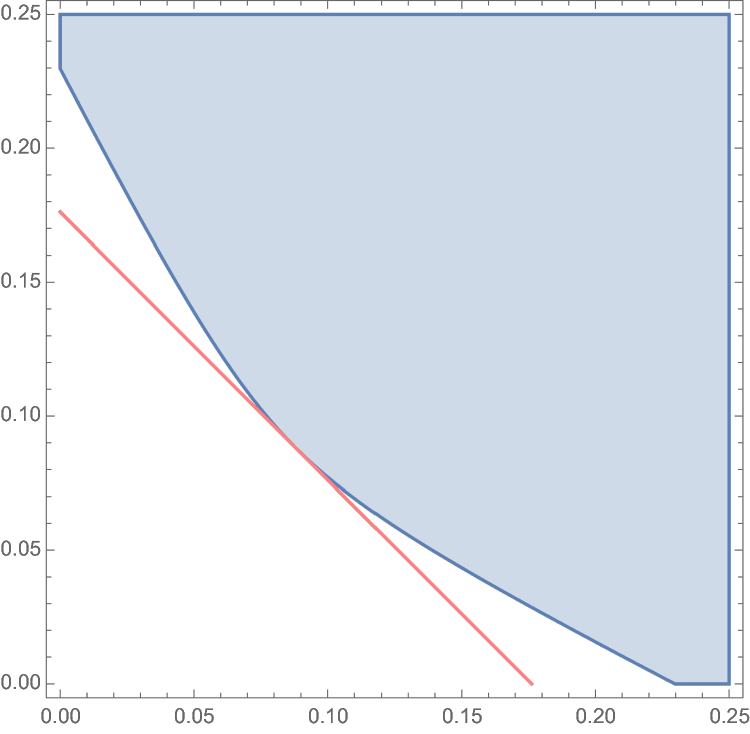}};
    \node at (6,0) {\includegraphics[scale=0.4]{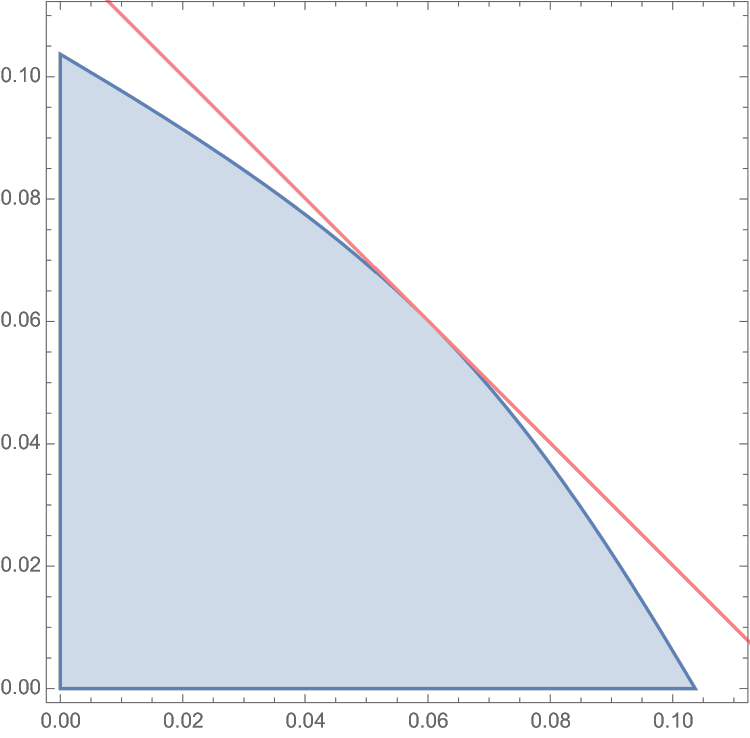}};
\end{tikzpicture}
    \caption{Optimizing $\lambda_2(w)$ (left) and $\lambda_n(w)$ (right) on the distance-$2$ graph of the Klein graph, after symmetry reduction.}
     \label{fig:KleinDistance2}
\end{figure}
Similarly, Problem~\eqref{eq:prob ln sdp} is the  $2$-dimensional convex optimization problem of maximizing $84 w_1 + 84 w_2$ subject to the nonpositivity of the eigenvalues and $w_1, w_2 \geq 0$. It has 
optimal solution $w_1 = w_2 = 1/\lambda_n(\ones) \sim 0.0600754$. The feasible region and the optimal position of the objective function are shown in Fig.~\ref{fig:KleinDistance2}.

\subsection{Complementary Slackness} Proposition~\ref{prop:certificates} provides certificates $X$ and $Y$ for the conformal rigidity of a graph. This approach can be 
especially effective. The conditions $L X = \lambda_2 X$ and $L Y = \lambda_n Y$ allows us to parametrize $X$ and $Y$ and reduce the number of variables. We explain 
the parametrization of $X$; the parametrization of $Y$ is similar.

\begin{lemma} \label{lem:cs parametrization} 
Set $U := [u_1 \ldots u_{k}] \in \RR^{n \times k}$ where $\{u_1, \ldots, u_k\}$ is a set of $k$ vectors in $\mathbb{R}^n$ that form a basis for the eigenspace $\mathcal{E}_{\lambda_2}$ of $L$. Then  $X \in \mathcal{S}^n_+$ and $LX = \lambda_2 X$ if and only if $S \in \mathcal{S}^k_+$ and $X = USU^\top$.
\end{lemma}
\begin{proof}
    If $X = USU^\top$ and $S \succeq 0$, then $X \succeq 0$ since for all $v \in \RR^n$
    $$v^\top X v = v^\top U S U^\top v = (U^\top v)^\top S (U^\top v) \geq 0.$$
    The columns of $X$ are combinations of the columns of $U$ and so, $LX = \lambda_2 X$. 
    
    Now suppose $X \succeq 0$ and $LX = \lambda_2 X$. Then $\rank(X) \leq k$ where $k = \textup{mult}(\lambda_2)$. Since $X \succeq 0$, $X$ has a reduced singular value decomposition of the form  $X = A D A^\top$ where the columns of $A$ are orthonormal and $D \in \RR^{k \times k}$ is a diagonal matrix with nonnegative diagonal entries. The column space of $X$ is the column space of $A$ which means that $A = UB$ for some matrix $B \in \RR^{k \times k}$. Therefore, 
    $$ X = UBD(UB)^\top = U (BDB^\top) U^\top 
 = U S U^\top.$$
 The matrix $S = BDB^\top$ is psd since 
 $D$ is psd. 
\end{proof}
Note that in Lemma~\ref{lem:cs parametrization} we do not require that the columns of $U$ 
form an orthonormal basis of $\mathcal{E}_{\lambda_2}$, just that it is a basis. 
Using Lemma~\ref{lem:cs parametrization}, we can rephrase Proposition~\ref{prop:certificates} as follows.

\begin{corollary} \label{cor:reduced certificates}
Let $\{u_1, \ldots, u_k\}$ be a basis of the eigenspace of $L$ with eigenvalue $\lambda_2$ and $U = [ u_1 \cdots u_k] \in \RR^{n \times k}$. Let 
$\{v_1, \ldots, v_\ell\}$ be a basis of the eigenspace of $L$ with eigenvalue $\lambda_n$ and $V = [ v_1 \cdots v_\ell] \in \RR^{n \times \ell}$. Then $G$ is conformally rigid if and only if there exists 
\begin{enumerate}
    \item a matrix $X = USU^\top$ such that 
    \begin{align} \label{eq:reduced X cert}
    S \in \mathcal{S}^k_+, \,\, \,\,X_{ii} + X_{jj} - 2X_{ij} = 1 \,\,\forall ij \in E, \end{align}
\item a matrix $Y = VTV^\top$ such that 
    \begin{align} \label{eq:reduced Y cert}
    T \in \mathcal{S}^\ell_+, \,\, \,\,Y_{ii} + Y_{jj} - 2Y_{ij} = 1 \,\,\forall ij \in E. 
    \end{align}
\end{enumerate}
\end{corollary}

\begin{proof}
    By construction, $\ones^\top X \ones = \ones^\top U S U^\top \ones = 0$ since the columns of $U$ are orthogonal to 
    $\ones$. Similarly for $Y$. All other constraints in Proposition~\ref{prop:certificates} are satisfied by $X$ (and $Y$) under the given parametrization.
\end{proof}

\begin{example} \label{ex:complement Shrikhande}
    We can use the above method to prove the conformal rigidity of the Hoffman graph 
(see Fig.~\ref{fig:conformally rigid}) which has $16$ vertices and $32$ edges. Its Laplacian eigenvalues with multiplicities 
are:
$$ (\lambda, \,\textup{mult}(\lambda)): \,\,(8, 1), (6, 4), (4, 6), (2, 4), (0, 1).$$
The rows of the following matrix span $\mathcal{E}_2$. Hence, $U^\top$ is 
$$\small{\left(
\begin{array}{cccccccccccccccc}
 -1 & 1 & -1 & -1 & -1 & 1 & 1 & 1 & 0 & 0 & 0 & -2 & 2 & 0 & 0 & 0 \\
 -1 & -1 & 1 & 1 & -1 & 1 & 1 & -1 & 0 & 0 & -2 & 0 & 0 & 2 & 0 & 0 \\
 -1 & -1 & 1 & -1 & 1 & 1 & -1 & 1 & 0 & -2 & 0 & 0 & 0 & 0 & 2 & 0 \\
 -1 & -1 & -1 & 1 & 1 & -1 & 1 & 1 & -2 & 0 & 0 & 0 & 0 & 0 & 0 & 2 \\
\end{array}
\right)}$$
Setting $S$ to be a symbolic symmetric matrix of size $4 \times 4$, and $X = USU^\top$ 
we can solve the linear equations $X_{ii} + X_{jj}-2X_{ij}=1$ for all $ij \in E$. Plugging the solution back into $S$ we get the matrix: 
$$
S = 
\left(
\begin{array}{cccc}
 s_{11} & 0 & 0 & 0 \\
 0 & s_{22} & 0 & 0 \\
 0 & 0 & s_{33} & 0 \\
 0 & 0 & 0 & 1-s_{11}-s_{22}-s_{33} \\
\end{array}
\right)
$$
which is psd if and only if all its diagonal entries are nonnegative. Therefore, there are infinitely many choices for the certificate $X$. Picking  $s_{11}=s_{22}=s_{33}=1/4$, we get the certificate $X$ in Fig.~\ref{fig:cs Hoffman}.
The eigenspace of $\mathcal{E}_8$ is spanned by 
$$(-1, -1, -1, -1, -1, -1, -1, -1, 1, 1, 1, 1, 1, 1, 1, 1),$$ and running the 
same procedure, we get the unique certificate $Y$ in Fig.~\ref{fig:cs Hoffman}.
    \begin{center}
    \begin{figure}[h!]
    \begin{tikzpicture}
            \node at (0,0) {\includegraphics[width=0.3\textwidth]{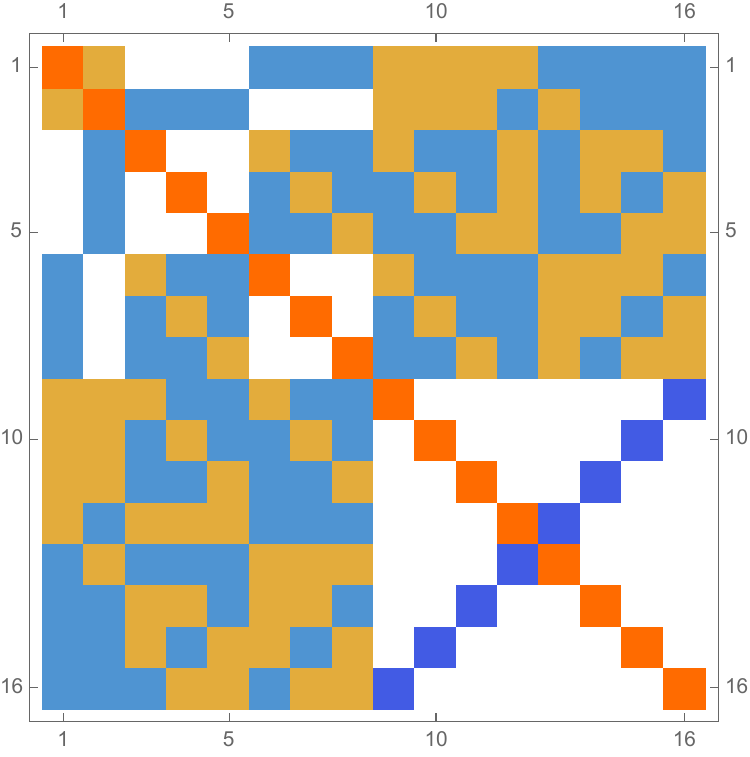}};
          \node at (4,0) {\includegraphics[width=0.3 \textwidth]{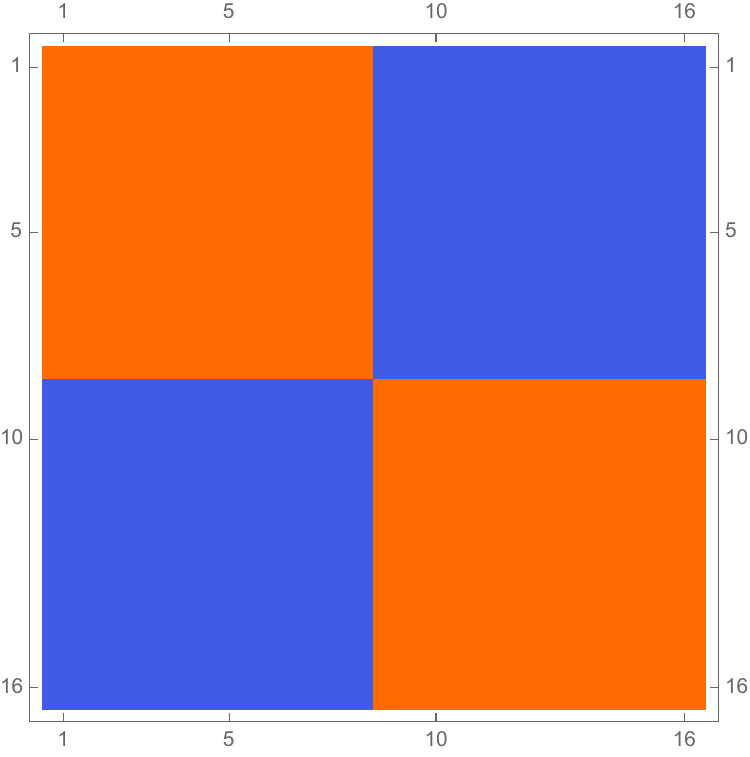}};
    \end{tikzpicture}
    \caption{Certificates for the conformal rigidity of the Hoffman graph. Left: a certificate $X$ for $\lambda_2$, Right: the certificate $Y$ for $\lambda_n$.}
    \label{fig:cs Hoffman}
    \end{figure}
\end{center}
\end{example}

    We use Corollary~\ref{cor:G and Gc cr} and the above parametrizations to show that while $C_6$ is conformally rigid, 
its complement is not. 

\begin{example} \label{ex:C6 and its complement}
    The cycle $C_6$ is conformally rigid since it is edge-transitive. Let 
    $G = C_6^c$ which is not edge-transitive and has $9$ edges. Its Laplacian has eigenvalues 
    $0,2,3,3,5,5$. 
    The eigenspace $\mathcal{E}^c_5$ is spanned by the orthonormal vectors 
    \begin{align*} 
    u = (-1/(2 \sqrt{3}), -1/ \sqrt{3},
    -1/(2 \sqrt{3}), 1/(2 \sqrt{3}),  1/\sqrt{3},  1/(2 \sqrt{3}),\\
    v = (1/2,0, -1/2, -1/2, 0,1/2).
    \end{align*}
  %   and the eigenvector of $\lambda_6 = 4$ is 
  % $$w = ( -1/ \sqrt{6}, 1/ \sqrt{6}, -1/ \sqrt{6}, 1/ \sqrt{6}, -1/ \sqrt{6}, 1/  \sqrt{6})
  
  \noindent Setting  
  $$ 
  % S = \begin{pmatrix} a & b \\ b & c \end{pmatrix}, 
  % T = \begin{pmatrix} d \end{pmatrix}, 
  % S' = \begin{pmatrix} e \end{pmatrix}, 
  T = \begin{pmatrix} a & b\\ b & c \end{pmatrix},$$
consider a potential certificate $Y'=[u \, v] \,T \,[u,v]^\top$. 
If $C_6^c$ were conformally rigid, we would have  
$\textup{Trace } Y' = \frac{9}{5} = \textup{Trace }(T) = a+c$. Therefore, 
we can set $c = 9/5-a$ in $Y'$. Then imposing the linear conditions, we get from 
the edge $(1,3)$ of $C_6^c$ that $a+\sqrt{3}b = 7/10$ and from the edge $(2,5)$ of $C_6^c$ that 
$a+\sqrt{3}b=3/5$ which is a contradiction. So no such $Y'$ exists and $C_6^c$ is not conformally rigid.\qed
\end{example}

\subsection{Using $UU^\top$} \label{sec:uut} We conclude by noting a particularly simple method that often ends up producing certificates $X$ and $Y$ (and provably does so for distance-regular graphs). For $\lambda_2$, the procedure works as follows: one collects an orthonormal basis of the eigenspace $\lambda_2$ corresponding to the Laplacian $L$ in a matrix $U$. A rescaling of the matrix $U U^\top \rightarrow c U U^{\top}$ to ensure that the trace is $|E|/\lambda_2$  often produces certificates that seem to work: in particular, running a check on the graphs built into the internal database of Mathematica shows that this approach leads to a valid certificate (for $\lambda_2$) for all conformally rigid graphs having less than 20 vertices. A 
%The first case where this approach does not work is the crossing number graph 6B on 20 vertices \cite{crossing} (see Fig. \ref{fig:conformally rigid}), which also appears in our list of sporadic outliers. Then there are further examples where this approach fails: 
case where this does not work is the Haar graph 565 on 20 vertices (see Fig. \ref{fig:Haar}). However, there are other examples such as the non-Cayley vertex-transitive graph (24,23) on $n=24$ vertices, where all the other approaches failed while $c U U^{\top}$ indeed produces a valid certificate for $\lambda_2$.

\textbf{Acknowledgment.} SS was supported by the NSF (DMS-2123224). RT acknowledges helpful discussions with Sameer Agarwal.

\end{document}